\documentclass[11pt,a4paper]{amsart}

\usepackage{amsmath, amsthm}              
\usepackage{pgfplots,pdflscape,comment}
\usepackage{amssymb,bm}
\usepackage{amsthm, wasysym,url}
\usepackage{enumitem}

\newtheorem{thm}{Theorem}[section]

\newtheorem{lemma}[thm]{Lemma}
\newtheorem{cor}[thm]{Corollary}

\theoremstyle{remark}
\newtheorem{remark}[thm]{{\bf Remark}}
\newtheorem{example}[thm]{{\bf Example}}

\usepackage{amsmath, amssymb, amsthm,color,tikz, mathdots}
\usetikzlibrary{arrows,decorations.pathmorphing,backgrounds,fit}
\newcommand{\hatt}[1]{%
\hat{#1}}

\newcommand{\doublehat}[1]{%
\bar{#1}}

\newenvironment{myproof}[1] {{\em Proof of {#1}. }}{\hfill$\square$}

\title{The inheritance of local bifurcations in mass action networks}

\author{Murad Banaji}
\address{University of Oxford}
\email{murad.banaji@maths.ox.ac.uk}

\author{Bal\'azs Boros}
\address{University of Vienna}
\email{balazs.boros@univie.ac.at}
\thanks{BB's work was supported by the Austrian Science Fund (FWF), project P32532.}

\author{Josef Hofbauer}
\address{University of Vienna}
\email{josef.hofbauer@univie.ac.at}

\begin{document}

\begin{abstract}

  We consider local bifurcations of equilibria in dynamical systems arising from chemical reaction networks with mass action kinetics. In particular, given any mass action network admitting a local bifurcation of equilibria, assuming only a general transversality condition, we list some enlargements of the network which preserve its capacity for the bifurcation. These results allow us to identify bifurcations in reaction networks from examination of their subnetworks, extending and complementing previous results on the inheritance of nontrivial dynamical behaviours amongst mass action networks. A number of examples are presented to illustrate applicability of the results.

  \smallskip
  \noindent \textbf{Keywords.} chemical reaction networks, mass action, bifurcations, transversality, inheritance
\end{abstract}

\maketitle

\section{Introduction}

A chemical reaction network (CRN) with mass action kinetics defines a parameterised family of ordinary differential equations (ODEs). There has been considerable recent interest in when such families, and indeed systems arising from reaction networks with other choices of kinetics, admit various nontrivial behaviours such as multistationarity and oscillation (\cite{craciun,banajicraciun2,domijan,joshishiu,banajipantea,muellerAll,banajiCRNosci,obatake,conradipantea2019,dickenstein2019,boroshofbauerplanar,boros:hofbauer:2022b,banajiborosnonlinearity}, for example). We will refer to a CRN with mass action kinetics, and also to the family of ODEs it generates, as a {\em mass action system}. In this paper we focus on mass action systems, although some of the results are easily extended to ODEs arising from reaction networks with other choices of kinetics. 

As bifurcations serve as organising centres for nontrivial dynamical behaviours, it is natural when seeking interesting dynamics in mass action systems to try and characterise those which admit particular bifurcations. Amongst small systems we may hope, with the help of computer algebra, to list those which admit some particular bifurcation, as done in \cite{banajiborosnonlinearity} in the case of Andronov--Hopf and Bautin bifurcations in small bimolecular systems. However, in larger systems, this direct approach rapidly becomes computationally difficult or impossible. 

One approach to finding interesting behaviours in mass action systems of arbitrary size involves so-called {\em inheritance} results (\cite{joshishiu,feliuwiufInterface2013,banajipanteaMPNE,banajiCRNosci,Cappelletti2020,banajiboroshofbauer,banajisplit}, for example). These tell us which enlargements of reaction systems preserve certain dynamical behaviours. The appeal of these results is that we can use them to infer interesting dynamics in a network by merely observing the presence of a certain subnetwork. Moreover, inheritance results give us information about regions in parameter space where the dynamical behaviour of interest must occur.

In this paper we write down an inheritance result for local bifurcations of equilibria in mass action systems. In particular, we begin with a mass action system, say $\mathcal{R}$, and suppose that, as we vary its rate constants, a positive equilibrium of $\mathcal{R}$ undergoes a local bifurcation, denoted by $B$. We now enlarge $\mathcal{R}$ to create a new CRN, say $\mathcal{R}'$, in natural ways which involve adding species and/or reactions to $\mathcal{R}$. Our goal is to list some conditions on the relationship between $\mathcal{R}$ and $\mathcal{R}'$, and on the bifurcation $B$, which suffice to guarantee that $\mathcal{R}'$ admits the same bifurcation $B$. In fact, we will show that if the enlargement belongs to one of six classes previously described \cite{banajisplit}, and if $B$ is unfolded transversely by the rate constants of $\mathcal{R}$ in the natural sense, then $\mathcal{R}'$ will admit bifurcation $B$, unfolded transversely by its rate constants. In this case, we say that the bifurcation $B$ occurring in $\mathcal{R}'$ is {\em inherited} from $\mathcal{R}$. These claims are the content of the main theorem, namely, Theorem~\ref{thmbifinherit}.

The main tools that we use to prove these claims are regular perturbation theory, and geometric singular perturbation theory (GSPT) in the form developed by Fenichel \cite{Fenichel79}. The challenge lies in setting up the problems so that we can apply the results of perturbation theory. Before presenting the details, we make a few general remarks. 

We note, first, that the state space of any model of a CRN is, in general, foliated by positively invariant polyhedra, termed {\em stoichiometric classes}. The intersection of any such class with the positive orthant is a {\em positive stoichiometric class}, and the ``rank'' of the network is the dimension of any positive stoichiometric class. Our assumption about $\mathcal{R}$ is that if we restrict attention to one of its positive stoichiometric classes and vary its rate constants, the local bifurcation $B$ is observed on this class. The conclusion we hope to arrive at about $\mathcal{R}'$ is formally identical. However, since $\mathcal{R}'$ may have greater rank than $\mathcal{R}$, some care is needed in interpreting what we mean by the claim that the ``same'' bifurcation occurs in $\mathcal{R}'$. In this case, we will show that, for some range of rate constants and on some stoichiometric class $\mathcal{S}'$ of $\mathcal{R}'$, there exists a locally invariant manifold $\mathcal{E} \subseteq \mathcal{S}'$ of dimension equal to the rank of $\mathcal{R}$; and on $\mathcal{E}$ we observe the bifurcation $B$. Moreover, directions in $\mathcal{S}'$ transverse to $\mathcal{E}$ are associated with exponentially attracting dynamics, merely adding to the dimension of the stable manifold of the bifurcating equilibrium and other nearby objects on the (parameter-dependent) centre manifold. 

Secondly, we note that in order to arrive at the conclusions we will not need to specify details of the bifurcation $B$, beyond the assumption that the bifurcation conditions can be phrased, locally, as finitely many smooth conditions on jets of the vector field at the bifurcating equilibrium (i.e., the vector field and its derivatives up to some finite order), and a general transversality hypothesis. Indeed, we do not need to assume that all the nondegeneracy conditions usually associated with any given bifurcation hold; however if they do, then they continue to  hold also for the enlarged network. As a concrete example, let us suppose that $\mathcal{R}'$ inherits bifurcations from $\mathcal{R}$: if an Andronov--Hopf bifurcation occurs in $\mathcal{R}$, and is unfolded transversely by the rate constants, then an Andronov--Hopf bifurcation must occur in $\mathcal{R}'$, again unfolded transversely by its rate constants; if, in fact, the Andronov--Hopf bifurcation in $\mathcal{R}$ is {\em supercritical}, then $\mathcal{R}'$ too admits a supercritical Andronov--Hopf bifurcation. These points will be illustrated in the examples in Section~\ref{secex}.

A third remark is that although here we restrict attention to local bifurcations of equilibria, the general framework we construct is considerably more flexible, and the results can be extended to a wider class of bifurcations, not necessarily local. 

In Section~\ref{secbif}, bifurcations and their persistence under regular and singular perturbations are discussed in a general setting. In Section~\ref{secmain}, the main theorem on bifurcations in CRNs is stated and proved. In Section~\ref{secex} we provide some examples, which demonstrate potential use of the results in this paper. We make a few concluding remarks in Section~\ref{secconc}.

\section{Local bifurcations of equilibria}
\label{secbif}

Before turning to mass action systems, we discuss local bifurcations of equilibria in a general context, although with some assumptions particularly suited to the applications to follow. We follow \cite{arnold2012geometrical,wiggins}, and the reader is referred to these references and \cite{HirschDifferentialTopology} for some of the details on bifurcation theory, jets of functions, and transversality. 

Throughout this section we will be concerned with the differential equation 
\begin{equation}
\label{eqgeneral}
\dot x = f(x, \kappa)\,,
\end{equation}
where $x$ varies in some open set $V_x \subseteq \mathbb{R}^n$, while the parameter $\kappa$ varies in an open set $V_\kappa \subseteq \mathbb{R}^m$. The parameter-dependent vector field, $f\colon V_x \times V_\kappa \to \mathbb{R}^n$, is assumed to be $C^k$ for some prescribed $k$, where the value of $k$ needed will depend on the bifurcation in question. We remark that although mass action systems define polynomial vector fields, the systems we ultimately examine may be derived from these via GSPT, and consequently may have only finite smoothness.

We write $D_xf(x,\kappa)$ for the derivative of $f$ with respect to $x$ evaluated at $(x,\kappa)$, and $D^{(j)}_xf(x,\kappa)$ for the $j$th derivative of $f$ with respect to $x$ evaluated at $(x,\kappa)$. We may identify $D_x^{(j)}$ with a symmetric $j$-linear map from $\mathbb{R}^n \times \cdots \times \mathbb{R}^n$ ($j$ times) to $\mathbb{R}^n$. For any fixed $k \geq 0$ and $\kappa \in V_\kappa$, the $k$-jet of $f(\cdot, \kappa)$ at $x \in V_x$ is defined as $J_x^kf(x, \kappa):=(x, f(x,\kappa),D_xf(x,\kappa), \ldots, D^{(k)}_xf(x,\kappa))$.  

We denote the set of all $k$-jets of all $C^k$ functions $V_x \to \mathbb{R}^n$ by $J^k(V_x, \mathbb{R}^n)$. Clearly $J^k(V_x, \mathbb{R}^n)$ is an open subset of $J^k(\mathbb{R}^n, \mathbb{R}^n)$ which is naturally identified with $\mathbb{R}^n \times \mathbb{R}^q$ for some $q$. We denote by $J^k_0(\mathbb{R}^n, \mathbb{R}^n)$ the subset of $J^k(\mathbb{R}^n, \mathbb{R}^n)$ where $f=0$, which is naturally identified with $\mathbb{R}^n \times \mathbb{R}^{q-n}$. For fixed $f\colon V_x \times V_\kappa \to \mathbb{R}^n$ and fixed $\kappa \in V_\kappa$, the map $x \mapsto J_x^{k}f(x, \kappa)$ which embeds a copy of $V_x$ into $J^k(\mathbb{R}^n, \mathbb{R}^n)$ is the ``$k$-jet extension'' of $f(\cdot, \kappa)$. We will denote the map which takes $(x, \kappa) \in V_x \times V_\kappa$ to $J_x^{k}f(x, \kappa)$ as $f^*$, namely $f^*(x, \kappa) = J_x^{k}f(x, \kappa)$. The image of $f^*$ is thus the union of the images of $f^*(\cdot, \kappa)$, for $\kappa \in V_\kappa$. Note that if $f$ is $C^{k+k'}$ ($k' \geq 0$), then $f^*$ is $C^{k'}$. 

We now consider some local bifurcation of equilibria, denoted $B$, of $C^k$ vector fields on any open subset of $\mathbb{R}^n$. We assume that $B$ can be defined as follows: there exists a smooth, embedded submanifold $\mathcal{B}^* \subseteq J^k_0(\mathbb{R}^n, \mathbb{R}^n)$, of codimension $c \geq 1$ in $J^k_0(\mathbb{R}^n, \mathbb{R}^n)$, such that given any $C^k$ vector field $F$ on some open subset of $\mathbb{R}^n$, $x$ is a {\em bifurcation point} of $F$ corresponding to the bifurcation $B$ if and only if $F^*(x) \in \mathcal{B}^*$ (where $F^*$ is the $k$-jet extension of $F$). This definition merely encodes the fact that $x$ is a nonhyperbolic equilibrium of $F$ whose nature is determined by smooth conditions on the $k$-jet of $F$ at $x$. Given a parameterised family such as in (\ref{eqgeneral}), we may also refer to  $(x, \kappa) \in V_x \times V_\kappa$ as a bifurcation point corresponding to $B$ if $f^*(x, \kappa) \in \mathcal{B}^*$. We define the codimension of $B$ to be the codimension of $\mathcal{B}^*$ in $J^k_0(\mathbb{R}^n, \mathbb{R}^n)$, namely, $c$.

We also assume that the occurrence of $B$ is independent of any choice of coordinates on the domain of the vector field. In other words, a sufficiently smooth change of coordinates cannot create or destroy a bifurcation point. More precisely, given vector fields $F,\hat{F}$ on open sets $U, V$ in $\mathbb{R}^n$, and a $C^{k+1}$ diffeomorphism $g$ between $U$ and $V$ such that $\hat{F}(y) = (Dg)(g^{-1}(y))F(g^{-1}(y))$, then $F^*(x) \in \mathcal{B}^*$ if and only if $\hat{F}^*(g(x)) \in \mathcal{B}^*$. In particular, choosing $g$ to be a translation, we see that $(x, d_1, d_2, \ldots) \in \mathcal{B}^*$ implies $(y, d_1, d_2, \ldots) \in \mathcal{B}^*$ for any $x,y \in \mathbb{R}^n$.

\begin{remark}[Characterising bifurcations via $k$-jets]
We do not claim that every local bifurcation is best characterised via conditions on the $k$-jet of the vector field at the bifurcation point. However, it is such infinitesimal conditions which are, in practice, most useful in confirming the occurrence of bifurcations in specific applications.
\end{remark}

It follows from the definition that given a bifurcation point $(\tilde{x}, \tilde{\kappa}) \in V_x \times V_\kappa$ of (\ref{eqgeneral}), there exists an open neighbourhood $W \subseteq V_x \times V_\kappa$ of $(\tilde{x}, \tilde{\kappa})$ and a function $h \colon W \to \mathbb{R}^{n} \times \mathbb{R}^c$ such that $(x, \kappa) \in W$ is a bifurcation point of (\ref{eqgeneral}) if and only if $h(x, \kappa)=0$. It is convenient to break down the vanishing of $h$ into two conditions, namely, the vanishing of the vector field, and the bifurcation conditions:

\begin{enumerate}
\item[(B1)] {\bf Equilibrium condition}. The vector field vanishes at $(x, \kappa)$, i.e., $f(x,\kappa)=0$. Equivalently, $f^*(x, \kappa) \in J^k_0(\mathbb{R}^n, \mathbb{R}^n)$.
\item[(B2)] {\bf Bifurcation conditions}. As $\mathcal{B}^*$ is a smooth embedded submanifold of $J^k_0(\mathbb{R}^n, \mathbb{R}^n)$ with codimension $c$, by standard arguments (e.g., Proposition~5.16 in \cite{LeeSmoothManifolds}) there exists a neighbourhood $W^*$ of $f^*(\tilde{x},\tilde{\kappa})$ and a smooth function $\tilde{G} \colon W^* \cap J^k_0(\mathbb{R}^n, \mathbb{R}^n) \to \mathbb{R}^c$ such that $\tilde{G}(y) = 0$ if and only if $y \in \mathcal{B}^*$ and, moreover, $\tilde{G}$ is regular at each point of $W^* \cap \mathcal{B}^*$. Shrinking $W^*$ if necessary, we may clearly extend $\tilde{G}$ to a smooth function $G\colon W^* \to \mathbb{R}^c$. Setting $W := (f^*)^{-1}(W^*) \cap (V_x \times V_\kappa)$, we may define the bifurcation function $g\colon W \to \mathbb{R}^c$ by
\[
g(x,\kappa) := G(x, f(x,\kappa),D_xf(x,\kappa), \ldots, D^{(k)}_xf(x,\kappa))\,.
\]
\end{enumerate}
Thus, setting $\mathcal{B} := (f^*)^{-1}(\mathcal{B}^*) \cap W$, the bifurcation $B$ occurs at $(x, \kappa) \in W$ if and only if $(x, \kappa) \in \mathcal{B}$. In other words, there exists a neighbourhood $W$ of $(\tilde{x}, \tilde{\kappa})$, such that the bifurcation $B$ is defined in $W$ via the vanishing of the function $h := (f,g)\colon W \to \mathbb{R}^n \times \mathbb{R}^c$.

\begin{remark}[Writing down bifurcation conditions explicitly]
It can be nontrivial to write down bifurcation functions such as $G$ and $g$ above explicitly, or compute the zero locus of these functions for a given family of systems. As an example, checking the vanishing of the first Lyapunov coefficient (also sometimes termed the first ``focal value'') at an equilibrium with a simple pair of nonzero imaginary eigenvalues, in order to confirm a Bautin bifurcation \cite{kuznetsov:2023}, is computationally challenging even in $\mathbb{R}^3$ -- see, for example, the computations carried out in \cite{banajiborosnonlinearity}. As we will see, the results in this paper can reduce the need for such computations when we are analysing CRNs.
\end{remark}

\begin{remark}[Inequalities in bifurcation conditions]
  Apart from the defining equalities, bifurcation conditions generally include additional strict inequalities involving smooth functions of the partial derivatives of the vector field at the nonhyperbolic equilibrium. We do not need to separately consider these: we may assume that the definition of $\mathcal{B}^*$ encompasses all such conditions, assumed to correspond to open subsets of $J^k(\mathbb{R}^n, \mathbb{R}^n)$. Note, however, that the value of $k$ required in the definition of $\mathcal{B}^*$ will in general depend on the additional conditions. For example, in a system on $\mathbb{R}^2$ undergoing an Andronov--Hopf bifurcation we want the trace of the Jacobian matrix at the nonhyperbolic equilibrium to be zero (the defining equality), and the determinant to be positive (an additional inequality). We may also require the nonvanishing of the first Lyapunov coefficient. Any such strict inequalities define an open submanifold of the bifurcation set defined via equalities, and so once we have assumed the inequalities to hold at some bifurcation point, the same is true nearby. In this case, if we do not require nonvanishing of the first Lyapunov coefficient, then we may take $k\geq 1$, whereas otherwise we require $k\geq 3$.
\end{remark}

It is conceptually helpful, as discussed in Chapter~2 of \cite{kuznetsov:2023}, to separate bifurcation and nondegeneracy conditions capturing the nature of a nonhyperbolic equilibrium in a single dynamical system, from transversality conditions capturing the behaviour of a family of systems. It is transversality which is the property robust to perturbations, key to the results to follow. 

Given $(\tilde{x}, \tilde{\kappa}) \in \mathcal{B}$, we say that the bifurcation $B$ at $(\tilde{x}, \tilde{\kappa})$ is {\em unfolded transversely by the parameters $\kappa$} if the following condition holds.
\begin{enumerate}
\item[(B3)] {\bf Transversality}. The map $h$ is regular at $(\tilde{x},\tilde{\kappa})$ and $D_xh(\tilde{x}, \tilde{\kappa})$ has rank $n$. By some basic linear algebra, equivalently, we may phrase this condition as follows: we can choose $c$ parameters, w.l.o.g. $\kappa_1, \ldots, \kappa_{c}$, such that $\underline{h}(x, \kappa_1, \ldots, \kappa_c) := h(x, \kappa_1, \ldots, \kappa_c, \tilde{\kappa}_{c+1}, \ldots, \tilde{\kappa}_m)$ is regular at $(\tilde{x},\tilde{\kappa}_1, \ldots, \tilde{\kappa}_c)$. Note that in order for this condition to make sense we assume that $f$ is at least $C^{k+1}$, and that $m \geq c$, i.e., that the total number of free parameters is at least equal to the codimension of the bifurcation. 
\end{enumerate}

\begin{remark}[Geometrical consequences of the bifurcation conditions]
Conditions B1--B3, although general, capture the intuitive notion of what it means for a bifurcation to have codimension $c$ and to be unfolded transversely by the parameters. Let $(\tilde{x},\tilde{\kappa})$ be a bifurcation point of (\ref{eqgeneral}), and suppose that parameters $\kappa_1, \ldots, \kappa_c$ unfold the bifurcation. Set $\underline{W}$ to denote $\{(x,\kappa) \in W\,:\,\kappa_{c+1} = \tilde{\kappa}_{c+1}, \ldots, \kappa_{m} = \tilde{\kappa}_{m}\}$ , and let $\underline{f}^* = f^*|_{\underline{W}}$. The assumptions imply that, shrinking $W$ if necessary, 
\begin{enumerate}
\item[(i)] $\mathcal{B}$ is an $(m-c)$-dimensional submanifold of $V_x \times V_{\kappa}$. 
\item[(ii)] The ``bifurcation set'', namely the projection of $\mathcal{B}$ onto the parameter space $V_{\kappa}$, is an $(m-c)$-dimensional submanifold of $V_{\kappa}$. 
\item[(iii)] $\underline{f}^*$ is transverse to $\mathcal{B}^*$, namely, the parameterised family of vector fields is transverse to the manifold corresponding to $B$ in the space of $k$-jets.
\end{enumerate}

The first claim follows, via the implicit function theorem, from the regularity of $h$ at $(\tilde{x},\tilde{\kappa})$. The second claim follows because, additionally, $D_xh(\tilde{x}, \tilde{\kappa})$ has rank $n$. To see the third claim, it suffices (shrinking $W$ if necessary) to show that $\mathrm{im}D\underline{f}^*(\tilde{x},\tilde{\kappa}_1, \ldots, \tilde{\kappa}_c)$ is transverse to $\mathcal{B}^*$. Note that $\underline{h} = H \circ \underline{f}^*$, where $H\colon W^* \subseteq \mathbb{R}^{n+q} \to \mathbb{R}^{n}\times \mathbb{R}^c$ maps $y = (x, f, D_xf, \ldots, D^{(k)}_xf)$ to $(f, G(y))$, so that $\mathcal{B}^* \cap W^* = H^{-1}(0)$. Via the chain rule, $D\underline{h}(\tilde{x},\tilde{\kappa}_1, \ldots, \tilde{\kappa}_c) = DH(\tilde{y})\circ D\,\underline{f}^*(\tilde{x},\tilde{\kappa}_1, \ldots, \tilde{\kappa}_c)$ where $\tilde{y} = f^*(\tilde{x},\tilde{\kappa})$. We already know that $DH(\tilde{y})$ has rank $n+c$ by the regularity of $\tilde{G}$ at $\tilde{y}$. Thus $D\underline{h}$ has rank $n+c$ at $(\tilde{x},\tilde{\kappa}_1, \ldots, \tilde{\kappa}_c)$ if and only if $\mathrm{rank}\,D \underline{f}^*(\tilde{x},\tilde{\kappa}_1, \ldots, \tilde{\kappa}_c) = n+c$ and $\mathrm{im}D\underline{f}^*(\tilde{x},\tilde{\kappa}_1, \ldots, \tilde{\kappa}_c) \cap \mathrm{ker}DH(\tilde{y}) =\{0\}$, i.e., $\mathrm{im}D\underline{f}^*(\tilde{x},\tilde{\kappa}_1, \ldots, \tilde{\kappa}_c) \oplus \mathrm{ker}DH(\tilde{y}) \cong \mathbb{R}^{n+q}$. In other words, noting that $\mathrm{ker}DH(\tilde{y})$ is the tangent space to $\mathcal{B}^*$ at $\tilde{y}$, $\mathrm{im}D\underline{f}^*(\tilde{x},\tilde{\kappa}_1, \ldots, \tilde{\kappa}_c)$ is transverse to $\mathcal{B}^*$. 
\end{remark}

\subsection{Persistence of bifurcations in perturbed systems}

We now prove two basic lemmas on the persistence of bifurcations in perturbations of (\ref{eqgeneral}), in a form that is most useful for the results to follow. 

In Lemmas~\ref{lemregpert}~and~\ref{lemsingpert}, we assume that (\ref{eqgeneral}) undergoes the codimension-$c$ bifurcation $B$, unfolded transversely by the parameters $\kappa$, at some point $(\tilde{x},\tilde{\kappa}) \in V_x \times V_\kappa$, i.e., B1--B3 are all satisfied at $(\tilde{x},\tilde{\kappa})$. Let $k$ be the level of differentiability required in the definition of $B$. We also assume, w.l.o.g., that $V_\kappa$ is an open subset of $\mathbb{R}^c$, i.e., all but $c$ of the parameters have been fixed at their bifurcation values, and we leave free only $c$ parameters which unfold the bifurcation at $(\tilde{x}, \tilde{\kappa})$, so that $h:=(f,g) \colon W \subseteq \mathbb{R}^{n+c} \to \mathbb{R}^{n+c}$. Consequently, assumption B3 is now simply that $Dh$ is nonsingular at $(\tilde{x}, \tilde{\kappa})$.

\begin{lemma}[Persistence of bifurcations under regular perturbations]
\label{lemregpert}
Let $I\subseteq \mathbb{R}$ be an open interval with $0 \in I$, and consider the ODE
\begin{equation}
\label{eqregpert}
\dot x = \hat{f}(x, \kappa, \varepsilon)\,,
\end{equation}
where $(x,\kappa,\varepsilon) \in V_x \times V_\kappa \times I$, and $\hat{f}$ is $C^{k+s}$ ($s \geq 1$). We assume that (\ref{eqregpert}) is a regular perturbation of (\ref{eqgeneral}), namely, $\hat{f}(x,\kappa,0) = f(x,\kappa)$. Then, there exists $\varepsilon_0>0$ such that $(-\varepsilon_0, \varepsilon_0) \subseteq I$ and for any fixed $\varepsilon \in (-\varepsilon_0, \varepsilon_0)$, (\ref{eqregpert}) undergoes the bifurcation $B$, unfolded transversely by the parameters $\kappa$, at a point $(x(\varepsilon),\kappa(\varepsilon))$, where $x(\varepsilon)$ and $\kappa(\varepsilon)$ are $C^s$ and satisfy $x(0) = \tilde{x}$ and $\kappa(0) = \tilde{\kappa}$.
\end{lemma}

\begin{proof}
Define $\hat{g}(x,\kappa,\varepsilon) = G(x, \hat{f}(x,\kappa,\varepsilon), D_x\hat{f}(x,\kappa,\varepsilon), \ldots, D^{(k)}_x\hat{f}(x,\kappa,\varepsilon))$, and note that $\hat{g}(x,\kappa,0) = g(x,\kappa)$. Let $\hat{h}(x, \kappa, \varepsilon) := (\hat{f},\hat{g})(x, \kappa, \varepsilon)$ and note that $\hat{h}$ is $C^s$. By definition, for any fixed $\tilde\varepsilon \in I$, a bifurcation $B$ occurs for (\ref{eqregpert}) at $(x, \kappa) \in W$, and is unfolded transversely by the parameters $\kappa$, if $\hat{h}(x, \kappa, \tilde\varepsilon) = 0$ (conditions B1, B2), and $\hat{h}(\cdot, \cdot, \tilde\varepsilon)$ is regular at $(x, \kappa)$ (condition B3). 

By assumption, $D_{(x,\kappa)}\hat{h}|_{(\tilde{x},\tilde{\kappa},0)} = D_{(x,\kappa)}h|_{(\tilde{x},\tilde{\kappa})}$ is nonsingular, and so, by the implicit function theorem, we can find $\varepsilon_0 > 0$ such that for $\varepsilon \in (-\varepsilon_0, \varepsilon_0)$, there exist $C^s$ functions $x(\varepsilon)$, $\kappa(\varepsilon)$ satisfying $x(0) = \tilde{x}$ and $\kappa(0) = \tilde{\kappa}$ and $\hat{h}(x(\varepsilon), \kappa(\varepsilon), \varepsilon)=0$. Making $\varepsilon_0$ smaller if necessary, we can ensure that (i) $(x(\varepsilon), \kappa(\varepsilon)) \in W$, and (ii) $\hat{h}(\cdot, \cdot, \varepsilon)$ is regular at $(x(\varepsilon), \kappa(\varepsilon))$. The latter follows as $D_{(x,\kappa)}\hat{h}$ is continuous at $(\tilde{x},\tilde{\kappa},0)$. Thus, for each fixed $\varepsilon \in (-\varepsilon_0, \varepsilon_0)$, (\ref{eqregpert}) undergoes the bifurcation $B$ at $(x(\varepsilon), \kappa(\varepsilon),\varepsilon)$, unfolded transversely by the parameters $\kappa$. 
\end{proof}

We remark that in the case that $\hat{f}$ in Lemma~\ref{lemregpert} is smooth, then so are $\hat{g}$ and $\hat{h}$, and we may infer that $x(\varepsilon)$ and $\kappa(\varepsilon)$ are smooth.

\begin{lemma}[Persistence of bifurcations under singular perturbations]
\label{lemsingpert}
Let $I\subseteq \mathbb{R}$ be an open interval with $0 \in I$, let $n' \geq 1$, and let $V_y \subseteq \mathbb{R}^{n'}$ be open. Consider the ODE
\begin{equation}
\label{eqsingpert}
\begin{array}{rcl}
\dot x &=& \tilde{f}(x, y, \kappa, \varepsilon)\,,\\
\varepsilon \dot y & = & q(x, y, \kappa, \varepsilon)\,,\\
\dot \kappa & = & 0\,.
\end{array}
\end{equation}
Here $(x,y,\kappa,\varepsilon)\in V_x \times V_y \times V_\kappa \times I$, and $\tilde{f}$ and $q$ are $C^{k+s+1}$ ($s \geq 1$) functions on $V_x \times V_y  \times V_\kappa \times I$, taking values in $\mathbb{R}^n$ and $\mathbb{R}^{n'}$ respectively. We assume that (\ref{eqsingpert}) is a singular perturbation of (\ref{eqgeneral}), namely, if $q(x,y,\kappa,0)=0$ then $\tilde{f}(x,y,\kappa,0)= f(x,\kappa)$. Assume, moreover, that there exists $\tilde{y} \in V_y$ such that $q(\tilde{x},\tilde{y},\tilde{\kappa},0)=0$, and the eigenvalues of $D_{y}q(\tilde{x},\tilde{y},\tilde{\kappa},0)$ have negative real parts. Then:
\begin{enumerate}
\item[(i)] There exist $\varepsilon_0 >0$ with $(-\varepsilon_0, \varepsilon_0) \subseteq I$, neighbourhoods $U_{x} \subseteq V_x, U_y \subseteq V_y$ and $U_\kappa \subseteq V_\kappa$ of $\tilde{x}$, $\tilde{y}$ and $\tilde{\kappa}$ respectively, and a $C^{k+s}$ function $\psi\colon U_x \times U_\kappa \times (-\varepsilon_0,\varepsilon_0) \to U_y$ satisfying $\psi(\tilde{x}, \tilde{\kappa},0) = \tilde{y}$, and such that for any fixed $\varepsilon \in (-\varepsilon_0, \varepsilon_0)\backslash\{0\}$, $\mathcal{E}_\varepsilon := \{(x,\psi(x,\kappa,\varepsilon),\kappa)\,:\, (x,\kappa) \in U_x \times U_\kappa\}$ is locally invariant for (\ref{eqsingpert}). 
\item[(ii)] For each fixed $\varepsilon \in (-\varepsilon_0, \varepsilon_0)\backslash\{0\}$, the bifurcation $B$, unfolded transversely by the parameters $\kappa$, occurs in the restricted system $\dot x = \hat{f}(x,\kappa,\varepsilon) := \tilde{f}(x, \psi(x,\kappa,\varepsilon), \kappa, \varepsilon)$ at a point $(x(\varepsilon), \kappa(\varepsilon))$ where $x(\varepsilon)$ and $\kappa(\varepsilon)$ are $C^s$ and satisfy $x(0) = \tilde{x}$, $\kappa(0) = \tilde{\kappa}$. I.e., observing that $(x, \kappa)$ serves as a local coordinate on $\mathcal{E}_\varepsilon$, the bifurcation $B$ occurs in (\ref{eqsingpert}) restricted to $\mathcal{E}_\varepsilon$.

\item[(iii)] For each fixed $\varepsilon \in (0, \varepsilon_0)$, letting $(x(\varepsilon),\kappa(\varepsilon))$ be the bifurcation point of $\dot x = \hat{f}(x,\kappa,\varepsilon)$, and defining $y(\varepsilon) = \psi(x(\varepsilon), \kappa(\varepsilon), \varepsilon)$, the eigenvalues of (\ref{eqsingpert}) at $(x(\varepsilon), y(\varepsilon), \kappa(\varepsilon))$ corresponding to directions transverse to $\mathcal{E}_\varepsilon$ have negative real parts.

\end{enumerate}
\end{lemma}

\begin{proof}
Rescaling time in (\ref{eqsingpert}) in the usual way \cite{Fenichel79} gives the ``fast time'' system
\begin{equation}
\label{eqsingperta}
\begin{array}{rcl}
\dot x &=& \varepsilon\tilde{f}(x, y, \kappa, \varepsilon)\,,\\
\dot y & = & q(x, y, \kappa, \varepsilon)\,,\\
\dot \kappa & = & 0\,.
\end{array}
\end{equation}

Let $\mathcal{E}'_0$ be the zero set of $q$ when $\varepsilon=0$, namely $\mathcal{E}'_0 = \{(x,y,\kappa) \in V_x \times V_y \times V_\kappa\,:\,q(x,y,\kappa,0)=0\}$, and consider the point $(\tilde{x},\tilde{y},\tilde{\kappa}) \in \mathcal{E}'_0$. As $D_{y}q(\tilde{x},\tilde{y},\tilde{\kappa},0)$ is nonsingular, by the implicit function theorem we can, in a neighbourhood of $(\tilde{x},\tilde{y},\tilde{\kappa})$, represent $\mathcal{E}'_0$ as the graph of a $C^{k+s+1}$ function, say $\psi_0$, giving $y$ in terms of $x$ and $\kappa$ (i.e., such that $\psi_0(\tilde{x},\tilde{\kappa}) = \tilde{y}$). Let $U_{x} \subseteq V_x, U_y \subseteq V_y$ and $U_\kappa \subseteq V_\kappa$ be neighbourhoods of $\tilde{x}$, $\tilde{y}$ and $\tilde{\kappa}$ respectively such that 
\[
\mathcal{E}_0:= \mathcal{E}'_0 \cap (U_{x} \times U_y \times U_\kappa) = \{(x,\psi_0(x,\kappa),\kappa)\,:\, (x,\kappa) \in U_x \times U_\kappa\}\,.
\]
Fenichel's theory \cite{Fenichel79} now tells us that there exists $\varepsilon_1>0$ with $(-\varepsilon_1, \varepsilon_1) \subseteq I$ such that for each $\varepsilon \in (-\varepsilon_1,\varepsilon_1)$, (\ref{eqsingperta}) has a locally invariant, ``slow'' manifold $\mathcal{E}_{\varepsilon}$ close (on compact sets) to $\mathcal{E}_0$. In particular, shrinking $U_x$ and $U_\kappa$ if necessary, there exists a $C^{k+s}$ function $\psi\colon U_x \times U_\kappa \times (-\varepsilon_1,\varepsilon_1) \to U_y$ such that $\psi(x,\kappa,0) = \psi_0(x,\kappa)$ and 
\[
\mathcal{E}_\varepsilon := \{(x,\psi(x,\kappa,\varepsilon),\kappa)\,:\, (x,\kappa) \in U_x \times U_\kappa\}\,
\]
is locally invariant for (\ref{eqsingperta}) for any $\varepsilon \in (-\varepsilon_1,\varepsilon_1)$, and hence is locally invariant for (\ref{eqsingpert}) whenever $\varepsilon \in (-\varepsilon_1,\varepsilon_1)\backslash\{0\}$. In case $f$ was smooth, we may assume that $\psi$ is $C^r$ for any prescribed $r < \infty$, although we cannot infer that $\psi$ is smooth (see also Theorem~2 in \cite{JonesSingular}). For fixed $\varepsilon \in (-\varepsilon_1,\varepsilon_1)\backslash\{0\}$, system (\ref{eqsingpert}) restricted to $\mathcal{E}_\varepsilon$ can be written as
\begin{equation}
\label{eqsingpertlocal}
\dot x = \hat{f}(x,\kappa,\varepsilon) := \tilde{f}(x, \psi(x,\kappa, \varepsilon), \kappa, \varepsilon)\,.
\end{equation}
Here, we are treating $(x,\kappa)$ as local coordinates on $\mathcal{E}_\varepsilon$.

Note that $\hat{f}$ is $C^{k+s}$ on $U_x \times U_\kappa \times (-\varepsilon_1,\varepsilon_1)$ and the conditions of Lemma~\ref{lemregpert} are satisfied for (\ref{eqsingpertlocal}). By Lemma~\ref{lemregpert}, there exists $\varepsilon_0 \in (0,\varepsilon_1]$ such that for $\varepsilon \in (-\varepsilon_0, \varepsilon_0)$, there exist $C^s$ functions $x(\varepsilon)$, $\kappa(\varepsilon)$ satisfying $x(0) = \tilde{x}$ and $\kappa(0) = \tilde{\kappa}$ and such that the bifurcation $B$ occurs in the restricted system (\ref{eqsingpertlocal}) at $(x(\varepsilon), \kappa(\varepsilon))$, and is unfolded transversely by the parameters $\kappa$. In case $\tilde{f}$ and $q$ were smooth, we may choose $x(\varepsilon)$, $\kappa(\varepsilon)$ to be $C^r$ for any prescribed $r < \infty$. Setting $y(\varepsilon) = \psi(x(\varepsilon), \kappa(\varepsilon), \varepsilon)$, clearly $(x(\varepsilon), y(\varepsilon), \kappa(\varepsilon)) \in V_x \times V_y \times V_\kappa$.

Moreover, shrinking $\varepsilon_0$ if necessary, our assumption that the eigenvalues of $D_{y}q(\tilde{x},\tilde{y},\tilde{\kappa},0)$ have negative real parts implies that, for $\varepsilon \in (0, \varepsilon_0)$, eigenvalues of the bifurcating equilibrium $(x(\varepsilon), y(\varepsilon), \kappa(\varepsilon))$ of (\ref{eqsingpert}) in directions transverse to $\mathcal{E}_{\varepsilon}$ have negative real parts. These directions thus do not affect the nature of the bifurcation, merely adding to the dimension of the nonhyperbolic equilibrium's stable manifold.  
\end{proof}

\begin{remark}[Lifting bifurcations from $\mathbb{R}^n$ to $\mathbb{R}^{n+n'}$]
  \label{remB1}
  In a slight abuse of terminology, we will refer to the conclusion of Lemma~\ref{lemsingpert} by saying that for each fixed $\varepsilon \in (0,\varepsilon_0)$ (\ref{eqsingpert}) undergoes the bifurcation $B$ unfolded transversely by the parameters $\kappa$. This should be read to mean that bifurcation $B$, unfolded transversely by the parameters $\kappa$, occurs when we restrict attention to some exponentially attracting, locally invariant, $C^{k+s}$ submanifold of $V_x \times V_y \times V_\kappa$ of the same dimension as $V_x \times V_{\kappa}$, namely, $\mathcal{E}_{\varepsilon}$. In this paper we take this as the {\em definition} of what we mean when we claim that a given bifurcation $B$, originally defined for systems on $\mathbb{R}^n$, occurs on $\mathbb{R}^{n+n'}$ where $n' > 0$. This is analogous to assuming that two bifurcations at nonhyperbolic equilbria are equivalent if they are defined by the same conditions on derivatives when we restrict attention to a local centre manifold \cite{carr2012applications}, with the further restriction that additional directions correspond to eigenvalues with negative real parts. 
\end{remark} 

\begin{example}[A cusp bifurcation in $\mathbb{R}$ and $\mathbb{R}^2$]
  For a system on $\mathbb{R}$, the basic conditions for a cusp bifurcation \cite{kuznetsov:2023} without any assumptions of nondegeneracy correspond to the set $\mathcal{B}^* \subseteq J^2(\mathbb{R},\mathbb{R})$ defined by $f=f_x=f_{xx} = 0$. Correspondingly in $\mathbb{R}^2$, assuming attracting dynamics transverse to the centre manifold, we would have $\hat{\mathcal{B}}^* \subseteq J^2(\mathbb{R}^2,\mathbb{R}^2)$ defined by $\{(x,y,f,g,\ldots)\,:\,f=g=0,f_xg_y-f_yg_x=0,f_x+g_y < 0, \langle p, L(q,q)\rangle = 0\}$. Here subscripts denote partial differentiation (e.g. $f_x:= \partial f/\partial x$); $p$ and $q$ are (left and right) eigenvectors corresponding to the zero eigenvalue of the Jacobian matrix $D_{x,y}(f,g)$ which can be chosen to depend smoothly on the first derivatives of $f$ and $g$; and $L$ is the second derivative of $(f,g)$ regarded as a multilinear map $\mathbb{R}^2 \times \mathbb{R}^2 \to \mathbb{R}^2$ (see Section~5.4 in Kuznetsov \cite{kuznetsov:2023}). The fact that $p$ and $q$ can be chosen to depend smoothly on the entries of $D_{x,y}(f,g)$ relies on the fact that $0$ is a simple eigenvalue of $D_{x,y}(f,g)$, and thus the final condition is well defined and smooth only in a neighbourhood of points satisfying the remaining conditions.
\end{example}

\begin{remark}[Special structures in CRNs and generalisation of the results]
\label{remgeneralise}
We will see, in the proof of the main theorem on CRNs to follow, that the function $\psi_0$ appearing in the proof of Lemma~\ref{lemsingpert} is globally defined (without need to call on the implicit function theorem); and eigenvalues of $D_{y}q(\cdot, \cdot, \cdot,0)$ are real and negative throughout the graph of $\psi_0$. These observations potentially allow generalisation of the results to the case of global bifurcations, and bifurcations of objects other than equilibria. 
\end{remark}

\section{The inheritance of bifurcations in CRNs}
\label{secmain}

\subsection{Preliminaries on reaction networks and mass action systems}
We gather some notation and terminology on CRNs and mass action systems in brief.

Let $\mathbb{R}_+$ denote the positive real numbers, and $\mathbb{R}^n_+$ denote the positive orthant in $\mathbb{R}^n$. Given a vector of indeterminates $(x_1, \ldots, x_n)^\mathrm{t}$ and any $\alpha \in \mathbb{R}^n$, we write $x^\alpha$ as an abbreviation for the (generalised) monomial $x_1^{\alpha_1}\cdots x_n^{\alpha_n}$, and given $A \in \mathbb{R}^{m \times n}$, $x^A$ is an abbreviation for the vector of monomials whose $j$th element is $\prod_i x_i^{A_{ji}}$. Given a scalar $\varepsilon$, we write $\bm{\varepsilon}$ for the vector each of whose entries is $\varepsilon$, where the length of this vector is assumed to be clear from the context. Similarly $\bm{1}$ refers to a vector of $1$s whose length is inferred from the context. The symbol ``$\circ$'' will denote the entrywise product of matrices or vectors of the same dimensions.

We begin with a collection of (chemical) species, which can be regarded as an abstract finite set. A {\em complex} is a formal linear combination of species, where the coefficients are nonnegative integers. We denote by $\mathsf{0}$ the empty complex. A chemical reaction is the process by which one complex, termed the {\em reactant complex}, is converted into another, termed the {\em product complex}, or more formally an ordered pair of complexes. Given complexes $C$ and $D$, we write $C \rightarrow D$ for the reaction with reactant complex $C$ and product complex $D$. Given any species $\mathsf{Y}$, a reaction of the form $\mathsf{0} \rightarrow \mathsf{Y}$ (resp., $\mathsf{Y} \rightarrow \mathsf{0}$) is called an {\em inflow} (resp., {\em outflow}) reaction. A CRN is a finite collection of chemical reactions on a given set of species. A CRN is {\em fully open} if it includes the reactions $\mathsf{0} \rightarrow \mathsf{X}_i$ and $\mathsf{X}_i \rightarrow \mathsf{0}$ for each of its species $\mathsf{X}_i$.

Let us consider a given CRN, say $\mathcal{R}$, involving $m$ reactions on $n$ species $\mathsf{X}_1, \ldots, \mathsf{X}_n$, denoted collectively by $\mathsf{X}$. The reactant and product complexes of each reaction of $\mathcal{R}$ are naturally associated with (nonnegative integer) vectors, say $a,b \in \mathbb{R}^n$, termed the {\em reactant vector} and {\em product vector} of the reaction; and we may write the corresponding complexes as $a \cdot \mathsf{X},\, b\cdot \mathsf{X}$, and the reaction as $a \cdot \mathsf{X} \rightarrow b \cdot \mathsf{X}$. A CRN is termed {\em bimolecular} if the sum of entries in each of its reactant and product vectors never exceeds two. 

Given the chemical reaction $a \cdot \mathsf{X} \rightarrow b \cdot \mathsf{X}$, the vector $b-a \in \mathbb{R}^n$ is termed the {\em reaction vector} of the reaction. Choosing an ordering on the reactions of $\mathcal{R}$, we may write the reactant vectors as the columns of the {\em reactant matrix}, say $\Gamma_l$, and write the reaction vectors as the columns of the {\em stoichiometric matrix}, say $\Gamma$. The intersection of cosets of $\mathrm{im}\,\Gamma$ with $\mathbb{R}^n_+$ are the {\em positive stoichiometric classes} of the network. The {\em rank} of the CRN is defined as the rank of $\Gamma$, i.e., the dimension of the positive stoichiometric classes, which is clearly independent of the chosen orderings on species and reactions. 

In order to give $\mathcal{R}$ {\em mass action kinetics}, we require an additional positive vector, say $\kappa \in \mathbb{R}^m_+$, where $\kappa_i$ is the {\em rate constant} of the $i$th reaction. The network then defines a system of ordinary differential equations on the positive orthant in $\mathbb{R}^n$ of the form
\begin{equation}
  \label{eqMA}
\dot x = \Gamma (\kappa \circ x^{\Gamma_l^\mathrm{t}})\,.
\end{equation}
Here $x \in \mathbb{R}^n_+$ is the vector of species concentrations with $x_i$ being the concentration of $\mathsf{X}_i$, $\Gamma$ is the $n \times m$ stoichiometric matrix of the network, and $\Gamma_l$ is the $n \times m$ reactant matrix. We remark that in this paper we are only interested in evolution on the positive orthant, and hence we assume throughout that $x$ is a strictly positive vector. In a harmless overloading of terminology, we will refer to either a CRN with an associated vector of rate constants, or the corresponding differential equation (\ref{eqMA}), as a {\em mass action system}. When denoting a mass action system, we often put the rate constant associated with a reaction above the corresponding reaction, e.g., $C \overset{\varepsilon}{\rightarrow} D$ denotes the conversion of complex $C$ to complex $D$ with rate constant $\varepsilon$. 

\subsection{The main theorem}
We are now ready to state and prove the main theorem. We consider six possible ``elementary'' enlargements of a CRN $\mathcal{R}$, previously gathered and discussed in \cite{banajiborosnonlinearity}:
\begin{enumerate}
\item[E1.] {\em A new linearly dependent reaction.} We add to $\mathcal{R}$ a new reaction involving only existing chemical species of $\mathcal{R}$, and in such a way that the rank of the network is preserved.
\item[E2.] {\em The fully open extension.} We add in (if absent) all chemical reactions of the form $\mathsf{0} \rightarrow \mathsf{X}_i$ and $\mathsf{X}_i \rightarrow \mathsf{0}$ for each chemical species $\mathsf{X}_i$ of $\mathcal{R}$.
\item[E3.] {\em A new linearly dependent species.} We add a new chemical species into some nonempty subset of the reactions of $\mathcal{R}$, in such a way that the rank of the network is preserved.
\item[E4.] {\em A new species and its inflow-outflow.} We add a new chemical species, say $\mathsf{Y}$, into some subset (perhaps empty) of the reactions of $\mathcal{R}$, and also add the inflow and outflow reactions $\mathsf{0} \rightarrow \mathsf{Y}$ and $\mathsf{Y} \rightarrow \mathsf{0}$. 
\item[E5.] {\em New reversible reactions involving new species.} We add into $\mathcal{R}$ some new reversible reactions involving at least as many new species. Moreover, the new species figure nontrivially in the enlarged CRN in the sense that the submatrix of the new stoichiometric matrix corresponding to the added species has rank equal to the number of new reversible reactions.
\item[E6.] {\em Splitting reactions.} We split some reactions of $\mathcal{R}$ and insert complexes involving at least as many new species as the number of reactions split. Moreover, the new species figure nontrivially in the enlarged CRN in the sense that the submatrix of the new stoichiometric matrix corresponding to the added species has rank equal to the number of reactions which are split.
\end{enumerate}

\begin{remark}
The nondegeneracy assumption associated with enlargements E5 and E6 is discussed further in Remark~2.4 of \cite{banajisplit}. Its necessity, at least in the case of enlargement E5, is illustrated by an example in \cite[Section~2.3]{Gutierrez2023}.
\end{remark}

We will say that a CRN ``admits the bifurcation $B$ unfolded transversely by its rate constants'' if, in local coordinates on some positive stoichiometric class, the bifurcation $B$ occurs and is unfolded transversely by the rate constants in the sense of Remark~\ref{remB1}. With this terminology in mind, the following is the main result of this paper.

\begin{thm}
\label{thmbifinherit}
Let $\mathcal{R}$ be a CRN which, with mass action kinetics, admits a local bifurcation $B$ unfolded transversely by its rate constants on the positive orthant. Let $\mathcal{R}'$ be a CRN obtained from $\mathcal{R}$ by any finite sequence of enlargements E1--E6 above. Then $\mathcal{R}'$, with mass action kinetics, admits the bifurcation $B$ unfolded transversely by its rate constants on the positive orthant. 
\end{thm}

The key to proving Theorem~\ref{thmbifinherit} is to show that each of the modifications E1 to E6 results in either a regular perturbation problem satifying the conditions of Lemma~\ref{lemregpert}, or a singular perturbation problem satisfying the conditions of Lemma~\ref{lemsingpert}. Many of the necessary calculations appear in previous work, but we present at least outlines of each calculation. 

\begin{myproof}{Theorem~\ref{thmbifinherit}}
The result clearly holds in general if it holds in the special cases where $\mathcal{R}'$ is constructed from $\mathcal{R}$ via one of the six elementary enlargements E1--E6 (see Remark~\ref{remB1}), and so it is these six special cases that we prove. 

{\bf The basic set-up.} We assume that the CRN $\mathcal{R}$ involves $n$ species $\mathsf{X}_1, \ldots \mathsf{X}_n$ (denoted collectively by $\mathsf{X}$) and $m$ reactions, giving rise to the mass action ODE system (\ref{eqMA}), namely $\dot x = \Gamma (\kappa \circ x^{\Gamma_l^\mathrm{t}})$. Here $x = (x_1, \ldots, x_n) \in \mathbb{R}^n_{+}$ denotes the vector of species concetrations, $\Gamma$ is the $n \times m$ stoichiometric matrix, $\Gamma_l$ is the $n \times m$ reactant matrix, and $\kappa \in \mathbb{R}^m_+$ is the vector of rate constants. For ease of notation, we define $v(x, \kappa) := \kappa \circ x^{\Gamma_l^\mathrm{t}}$. Denote the rank of $\Gamma$ by $r$. 

Let $c$ be the codimension of the bifurcation $B$, and assume that $\mathcal{R}$ undergoes the bifurcation $B$ at $x = \tilde{x}$ and $\kappa = \tilde{\kappa}$. Without loss of generality we may assume that $\kappa \in \mathbb{R}^c_{+}$, i.e., we leave free only $c$ of the rate constants which unfold the bifurcation and fix the remaining $m-c$ rate constants at their values at the bifurcation point.

Let $\Gamma_0$ be any $n \times r$ matrix whose columns form a basis of $\mathrm{im}\,\Gamma$ (the stoichiometric subspace), and define the $r \times m$ matrix $\Lambda$ by $\Gamma = \Gamma_0\Lambda $. We can define a local coordinate $\theta \in \mathbb{R}^{r}$ on the positive stoichiometric class of $\tilde{x}$ by $x = \tilde{x} + \Gamma_0\theta$. Then $\theta$ evolves according to
\begin{equation}
\label{eqlocal}
\dot \theta = \Lambda v(\tilde{x} + \Gamma_0 \theta, \kappa) =: f(\theta,\kappa)\,,
\end{equation}
and the bifurcation now occurs at $(\theta, \kappa) = (0, \tilde{\kappa})$. With the bifurcation function $g\colon W \subseteq \mathbb{R}^{r} \times \mathbb{R}^c \to \mathbb{R}^c$ associated with the bifurcation in a neighbourhood of $(0, \tilde{\kappa})$ defined as in Section~\ref{secbif}, our assumptions so far imply that (i) $f(0,\tilde{\kappa})=0$, (ii) $g(0, \tilde{\kappa}) = 0$, and (iii) $\mathrm{rank}\,(D_{(\theta, \kappa)}h(0,\tilde{\kappa})) = r+c$, where $h = (f,g)$.

In all of the six constructions below, $\varepsilon$ is a scalar parameter to be controlled, and the physically relevant situation corresponds to $\varepsilon > 0$. In each case, where there is a clear correspondence between a reaction in the original network and one in the enlarged network, the corresponding reaction rate is left unchanged, unless explicitly stated otherwise.

{\bf Enlargement E1.} We obtain the enlarged CRN $\mathcal{R}'$ by adding to $\mathcal{R}$ a reaction involving only the existing species of $\mathcal{R}$ whose reaction vector is a linear combination of existing reaction vectors. Let the reactant complex of the new reaction be $\alpha \cdot \mathsf{X}$. By the assumption of linear dependence, there exists $s \in \mathbb{R}^{r}$ such that the new reaction has reaction vector $\Gamma_0s$. If we give the new reaction mass action kinetics with rate constant $\varepsilon$, then by easy calculations (see \cite{banajiCRNosci}, proof of Theorem~1) the local coordinate $\theta$ on the positive stoichiometric class of $\tilde{x}$ now evolves according to
\begin{equation}
\label{eqlocal1}
\dot \theta = \Lambda v(\tilde{x} + \Gamma_0 \theta, \kappa) + \varepsilon s (\tilde{x} + \Gamma_0 \theta)^\alpha =: \hat{f}(\theta,\kappa,\varepsilon)\,.
\end{equation}
We see immediately that (\ref{eqlocal1}) is a regular perturbation of (\ref{eqlocal}) with $\hat{f}$ smooth and $\hat{f}(\theta,\kappa,0) = f(\theta,\kappa)$. By Lemma~\ref{lemregpert} there exists $\varepsilon_0>0$ such that for $\varepsilon \in (-\varepsilon_0, \varepsilon_0)$, there exist smooth functions $\theta(\varepsilon),\kappa(\varepsilon)$ satisfying $\theta(0)=0$, $\kappa(0) = \tilde{\kappa}$ and such that (\ref{eqlocal1}) undergoes the bifurcation $B$, unfolded transversely by the rate constants $\kappa$, at $(\theta(\varepsilon),\kappa(\varepsilon))$. We may assume, reducing $\varepsilon_0$ if necessary, that $x(\varepsilon) := \tilde{x} + \Gamma_0\theta(\varepsilon)$ and $\kappa(\varepsilon)$ are positive, and so the bifurcation occurs on the positive orthant and at positive parameter values. 

{\bf Englargement E2.} We modify $\mathcal{R}$ by including inflows and outflows for all the species involved to obtain a CRN $\mathcal{R}'$. In other words, we add to $\mathcal{R}$ the reactions $\mathsf{0} \rightarrow \mathsf{X}_i$ and $\mathsf{X}_i \rightarrow \mathsf{0}$ if absent. We choose the rate constants for the added reactions $\mathsf{0} \rightarrow \mathsf{X}_i$ and $\mathsf{X}_i \rightarrow \mathsf{0}$ to be, respectively, $\varepsilon \tilde{x}_i$ and $\varepsilon$, where if either reaction already exists in $\mathcal{R}$ this is understood to mean that we add to its rate constant this quantity. The evolution of the species concentrations is now governed by
\begin{equation}
\label{eqglobal2}
\dot x = \Gamma v(x, \kappa) + \varepsilon (\tilde{x}-x)\,.
\end{equation}
This choice of rate constants for the inflows and outflows ensures that the positive stoichiometric class of $\tilde{x}$ in $\mathcal{R}$, say $\mathcal{S}_0$, remains locally invariant (although it is not a positive stoichiometric class of $\mathcal{R}'$, unless it happens to be equal to the entire positive orthant). We calculate that (see \cite{banajiCRNosci}, proof of Theorem~2) the local coordinate $\theta$ on $\mathcal{S}_0$ evolves according to 
\begin{equation}
\label{eqlocal2a}
\dot \theta = \Lambda  v(\tilde{x} + \Gamma_0 \theta, \kappa) - \varepsilon \theta =: \hat{f}(\theta,\kappa,\varepsilon) \,.
\end{equation}
Equation (\ref{eqlocal2a}) is a regular perturbation of (\ref{eqlocal}), as $\hat{f}$ is smooth and satisfies $\hat{f}(\theta,\kappa,0) = f(\theta,\kappa)$. By Lemma~\ref{lemregpert}, there exists $\varepsilon_0>0$ such that for $\varepsilon \in (-\varepsilon_0, \varepsilon_0)$, there exist $\theta(\varepsilon),\kappa(\varepsilon)$ satisfying $\theta(0)=0$, $\kappa(0) = \tilde{\kappa}$ and such that (\ref{eqlocal2a}) undergoes the bifurcation $B$, unfolded transversely by the rate constants $\kappa$, at $(\theta(\varepsilon),\kappa(\varepsilon))$. We may assume, by reducing $\varepsilon_0$ if necessary, that $x(\varepsilon) := \tilde{x} + \Gamma_0\theta(\varepsilon)$ and $\kappa(\varepsilon)$ are positive, and so the bifurcation occurs on the positive orthant and at positive parameter values. Moreover, when $\varepsilon>0$, eigenvalues of (\ref{eqglobal2}) at $(x(\varepsilon),\kappa(\varepsilon))$ associated with directions transverse to $\mathcal{S}_0$ are all real and negative (in fact, all equal to $-\varepsilon$, see \cite{banajiCRNosci}, proof of Theorem~2), simply adding to the dimension of the stable manifold of the bifurcating equilibrium. 

{\bf Enlargement E3.} We add into the existing reactions of $\mathcal{R}$ a new species $\mathsf{Y}$ in such a way that we do not change the rank of the network. Let $y$ denote the concentration of $\mathsf{Y}$, and let $\alpha_j$ denote its stoichiometry in the reactant complex of reaction $j$ in $\mathcal{R}'$. Choose the rate constants of $\mathcal{R}'$ so that the rate of the $j$th reaction is $\varepsilon^{\alpha_j} v_j(x,\kappa) y^{\alpha_j}$ (i.e., we multiply the rate constant of the $j$th reaction by $\varepsilon^{\alpha_j}$). By the assumption of linear dependence, there exists some $s \in \mathbb{R}^n$ such that the stoichiometric matrix of $\mathcal{R}'$ is
\[
\left(\begin{array}{c}\Gamma\\s^\mathrm{t}\Gamma\end{array}\right)\,.
\]
Note that $\mathcal{R}'$ has a new linear conservation law: $y-s^\mathrm{t}x$ is constant along trajectories, and we choose the value of this constant to be $\varepsilon^{-1}$. I.e., we focus our attention on the positive stoichiometric class, say $\mathcal{S}_\varepsilon$, of $\mathcal{R}'$ which includes the point $(\tilde{x}, s^\mathrm{t}\tilde{x} +\varepsilon^{-1})$. We can calculate (see \cite{banajiboroshofbauer}, proof of Theorem~1) that a local coordinate $\theta \in \mathbb{R}^{r}$ on $\mathcal{S}_\varepsilon$, defined as usual by $x = \tilde{x} + \Gamma_0\theta$, evolves according to
\begin{equation}
\label{eqlocal3}
\dot \theta = \Lambda (v(\tilde{x}+\Gamma_0 \theta, \kappa) \circ (1+\varepsilon s^\mathrm{t} (\tilde{x}+\Gamma_0\theta))^\alpha) =: \hat{f}(\theta,\kappa,\varepsilon)\,.
\end{equation}
Here $\alpha = (\alpha_1, \ldots, \alpha_n)$. Moreover, $\hat{f}$ is smooth, and satisfies $\hat{f}(\theta,\kappa,0) = f(\theta, \kappa)$. By Lemma~\ref{lemregpert}, there exists $\varepsilon_0>0$ such that for $\varepsilon \in (-\varepsilon_0, \varepsilon_0)$, there exist smooth functions $\theta(\varepsilon),\kappa(\varepsilon)$ satisfying $\theta(0)=0$, $\kappa(0) = \tilde{\kappa}$ and such that (\ref{eqlocal3}) undergoes the bifurcation $B$, unfolded transversely by the rate constants $\kappa$, at $(\theta(\varepsilon),\kappa(\varepsilon))$. We may assume, reducing $\varepsilon_0$ if necessary, that for all $\varepsilon \in (0,\varepsilon_0)$, $x(\varepsilon) := \tilde{x} + \Gamma_0\theta(\varepsilon)$, $y(\varepsilon) = s^\mathrm{t}x(\varepsilon) +\varepsilon^{-1}$ and $\kappa(\varepsilon)$ are positive, and so the bifurcation occurs on the positive orthant and at positive parameter values.

{\bf Enlargement E4.} Let us modify $\mathcal{R}$ to obtain a new CRN $\mathcal{R}'$ by adding into some subset of the reactions of $\mathcal{R}$ a new species $\mathsf{Y}$ in an arbitrary way, and also including inflow and outflow reactions $\mathsf{0} \rightarrow \mathsf{Y}$ and $\mathsf{Y} \rightarrow \mathsf{0}$. We set the rate constants of these added flow reactions to be $\varepsilon^{-1}$. Let $\alpha_i$ be the stoichiometric coefficient of $\mathsf{Y}$ in the reactant complex of the $i$th reaction and let $\alpha = (\alpha_1, \ldots, \alpha_m)$. Let $s_i$ be the net change in the stoichiometry of $\mathsf{Y}$ in the $i$th reaction and let $s=(s_1, \ldots, s_m)$. For each $\varepsilon \neq 0$, the dynamics of $\mathcal{R}'$ is now governed by
\begin{equation}
\label{eqglobal4}
\begin{array}{rcl}\dot x &=& \Gamma (v(x, \kappa)\circ y^\alpha)\,,\\\varepsilon \dot y & = & \varepsilon s (v(x,\kappa)\circ y^\alpha) + (1-y)\,.\end{array}
\end{equation}
We consider the evolution on the positive stoichiometric class, say $\mathcal{S}'$, which includes the point $(\tilde{x},1)$. We can choose $(\theta, y)$ to be a local coordinate on $\mathcal{S}'$ where, as usual, $\theta$ is defined by $x = \tilde{x} + \Gamma_0\theta$. We then have, for the evolution on $\mathcal{S}'$,
\begin{equation}
\label{eqlocal4}
\begin{array}{rcl}\dot \theta &=& \Lambda  (v(\tilde{x}+\Gamma_0\theta,\kappa)\circ y^\alpha)\,,\\\varepsilon \dot y & = & \varepsilon s (v(\tilde{x}+\Gamma_0\theta,\kappa)\circ y^\alpha) + (1-y)\,,\\\dot \kappa & = & 0\,.\end{array}
\end{equation}
(We have introduced trivial dynamics for the parameter $\kappa$.) Defining
\[
q(\theta, y,\kappa, \varepsilon):=\varepsilon s (v(\tilde{x}+\Gamma_0\theta,\kappa)\circ y^\alpha) + (1-y)\,,
\]
observe that $q(\theta,y,\kappa,0)=0$ if and only if $y=1$. By inspection $D_{y}q(\theta,y,\kappa,0)=-1$ everywhere and hence certainly at $(\theta, y, \kappa) = (0,1,\tilde{\kappa})$. The conditions of Lemma~\ref{lemsingpert} are satisfied, and thus there exists $\varepsilon_0 >0$ such that for any fixed $\varepsilon \in (0, \varepsilon_0)$, (\ref{eqlocal4}) undergoes the bifurcation $B$ unfolded transversely by the parameters $\kappa$ at a point $(\theta(\varepsilon), y(\varepsilon), \kappa(\varepsilon))$ with $\theta(0) = 0$, $y(0) = 1$ and $\kappa(0) = \tilde{\kappa}$. The functions $\theta(\varepsilon), y(\varepsilon)$ and $\kappa(\varepsilon)$ can be chosen to be $C^r$ for any finite $r$ (see the proof of Lemma~\ref{lemsingpert}). By reducing $\varepsilon_0$ if necessary, we can ensure that $\kappa(\varepsilon)$, $x(\varepsilon) = \tilde{x} + \Gamma_0\theta(\varepsilon)$ and $y(\varepsilon)$ are all positive, and so the bifurcation $B$ occurring in (\ref{eqglobal4}) takes place on the positive orthant and at positive parameter values.

{\bf Enlargement E5.} We now create $\mathcal{R}'$ by adding $m' \geq 1$ new reversible reactions into $\mathcal{R}$ involving $m'+k$ new chemical species ($k \geq 0$) $\mathsf{Y}_1, \ldots, \mathsf{Y}_{m'+k}$ which we will denote collectively by $\mathsf{Y}$. Let the new reactions be:
\[
a_i\cdot \mathsf{X} +b_i\cdot \mathsf{Y} \rightleftharpoons a_i'\cdot \mathsf{X} + b_i'\cdot \mathsf{Y},\quad (i = 1, \ldots, m')\,.
\]
Here $a_i, a_i', b_i$ and $b_i'$ are nonnegative integer vectors of length $n$, $n$, $m'+k$ and $m'+k$ respectively. Define $a = (a_1| a_2|\cdots|a_{m'}) \in \mathbb{R}^{n \times m'}$, with $a' \in \mathbb{R}^{n \times m'}$, $b \in \mathbb{R}^{(m'+k) \times m'}$ and $b' \in \mathbb{R}^{(m'+k) \times m'}$ defined similarly. Define $\alpha = a'-a \in \mathbb{R}^{n \times m'}$ and $\beta = b'-b \in \mathbb{R}^{(m'+k) \times m'}$. By assumption the new species $\mathsf{Y}$ figure nondegenerately in the enlarged CRN in the sense that $\beta$ has rank $m'$.

Let $y_i$ denote the concentration of $\mathsf{Y}_i$, and define $y:=(y_1, \ldots, y_{m'+k})^{\mathrm{t}}$. Given that $\beta$ has rank $m'$, we can assume without loss of generality (i.e., by reordering the added species $\mathsf{Y}$ if necessary) that $\beta = \left(\begin{array}{c}\hatt{\beta}\\\doublehat{\beta}\end{array}\right)$, where $\hatt{\beta}$ is a nonsingular $m' \times m'$ matrix, and $\doublehat{\beta}$ is a $k \times m'$ matrix. If $k=0$, then $\doublehat{\beta}$ is empty. The vectors $\hatt{y} \in \mathbb{R}^{m'}$, $\doublehat{y} \in \mathbb{R}^k$, $\hatt{b} \in \mathbb{R}^{m'\times m'}$, $\hatt{b}' \in \mathbb{R}^{m'\times m'}$, $\doublehat{b} \in \mathbb{R}^{k\times m'}$ and $\doublehat{b}' \in \mathbb{R}^{k\times m'}$ are defined in the natural way, with $\doublehat{y}$, $\doublehat{b}$ and $\doublehat{b}'$ empty if $k=0$. Conventions on empty matrices and vectors used to cover the case where $k=0$ are given in \cite{banajiCRNosci1}.

We give the new reactions mass action kinetics, and choose rate constants so that these reactions have rate vector 
\[
q(x,y,\varepsilon) := \bm{\varepsilon}^{-\hat{b}^{\mathrm{t}}}\circ x^{a^{\mathrm{t}}}\circ y^{b^{\mathrm{t}}} - \bm{\varepsilon}^{-{(\hat{b}')}^{\mathrm{t}}}\circ x^{(a')^{\mathrm{t}}}\circ y^{{(b')}^{\mathrm{t}}}\,.
\]
Then $\mathcal{R}'$ evolves according to:
\begin{equation}
\label{eqglobal5}
\left(\begin{array}{c}\dot x\\\dot y\end{array}\right) = \left(\begin{array}{cc}\Gamma&\alpha\\0&\beta\end{array}\right)\left(\begin{array}{c}v(x,\kappa)\\q(x,y,\varepsilon)\end{array}\right)\,.
\end{equation}
We define $\delta := -(\doublehat{\beta}\hatt{\beta}^{-1})^{\mathrm{t}}$ and can confirm that $\delta^{\mathrm{t}}\hatt{y}+\doublehat{y}$ is constant along trajectories of (\ref{eqglobal5}) (see \cite{banajiCRNosci1}). We fix $\delta^{\mathrm{t}}\hatt{y}+\doublehat{y} = \mathbf{1}$, and let $\mathcal{S}'$ be the positive stoichiometric class of $\mathcal{R}'$ which includes the point $(x,\hat{y},\doublehat{y}) = (\tilde{x},0,\bm{1})$ in its closure. 

Define $\theta$ by $x = \tilde{x} + \Gamma_0\theta + \alpha\hat{\beta}^{-1}\hat{y}$, and define $w = \varepsilon^{-1}\hat{y}$. For any fixed $\varepsilon > 0$,  $(\theta, w)$ is a local coordinate on $\mathcal{S}'$. Define $\hat{q}(\theta,w, \varepsilon) = q(\tilde{x}  + \Gamma_0\theta + \varepsilon\alpha\hat{\beta}^{-1}w,(\varepsilon w, \mathbf{1}-\varepsilon\delta^{\mathrm{t}}w),\varepsilon)$ for $\varepsilon \neq 0$ sufficiently small. For ease of notation, define $z(\theta):=\tilde{x} + \Gamma_0\theta$. We can write out $\hat{q}(\theta,w,\varepsilon)$ in full:
\[
\hat{q}(\theta,w, \varepsilon) = (z(\theta)+\varepsilon\alpha\hatt{\beta}^{-1}w)^{a^{\mathrm{t}}}\circ w^{\hat{b}^{\mathrm{t}}}\circ(\mathbf{1}-\varepsilon \delta^{\mathrm{t}}w)^{\doublehat{b}^{\mathrm{t}}} - (z(\theta)+\varepsilon\alpha\hatt{\beta}^{-1}w)^{{(a')}^{\mathrm{t}}}\circ w^{{(\hat{b}')}^{\mathrm{t}}}\circ (\mathbf{1}-\varepsilon\delta^{\mathrm{t}}w)^{{(\doublehat{b}')}^{\mathrm{t}}}\,,
\]
and note that $\hat{q}(\theta,w, \varepsilon)$ has a smooth extension to $\varepsilon = 0$, with
\[
\hat{q}(\theta,w,0) = z(\theta)^{a^{\mathrm{t}}}\circ w^{\hat{b}^{\mathrm{t}}} - z(\theta)^{{(a')}^{\mathrm{t}}}\circ w^{{(\hat{b}')}^{\mathrm{t}}}\,.
\]

We thus have, for the evolution of $\theta$, $w$, and $\kappa$,
\begin{equation}
\label{eqlocal5}
\begin{array}{rcl}\dot \theta & = & \Lambda  v(\tilde{x} + \Gamma_0\theta+\varepsilon\alpha\hatt{\beta}^{-1} w,\kappa)\,,\\
\varepsilon\dot {w} & = & \hatt{\beta} \hat{q}(\theta, w, \varepsilon)\,,\\
\dot \kappa & = & 0\,.
\end{array}
\end{equation}
The equation $\hat{q}(\theta,w,0) = 0$ can be solved to get $w = z(\theta)^\gamma$, where $\gamma := -(\alpha\,\hatt{\beta}^{-1})^{\mathrm{t}}$. When we set $\varepsilon=0$ and $\hat{q}(\theta,w,0)=0$, the dynamics of $\theta$ in equation (\ref{eqlocal5}) reduces to that of $\theta$ in (\ref{eqlocal}), and thus (\ref{eqlocal5}) is a singular perturbation of (\ref{eqlocal}) in the sense of Lemma~\ref{lemsingpert}.

We next show that the eigenvalues of $\hat{\beta}D_{w}\hat{q}(\theta,w,0)$ are real and negative provided $z(\theta)$ is positive, which is certainly true for sufficiently small $|\theta|$. 

Defining
\[
T_1(\theta,w,\varepsilon):=(z(\theta)+\varepsilon\alpha\hatt{\beta}^{-1}w)^{a^{\mathrm{t}}}\circ w^{\hat{b}^{\mathrm{t}}}\circ(\mathbf{1}-\varepsilon \delta^{\mathrm{t}}w)^{\doublehat{b}^{\mathrm{t}}} 
\]
and
\[
T_2(\theta,w,\varepsilon):=(z(\theta)+\varepsilon\alpha\hatt{\beta}^{-1}w)^{{(a')}^{\mathrm{t}}}\circ w^{{(\hat{b}')}^{\mathrm{t}}}\circ (\mathbf{1}-\varepsilon\delta^{\mathrm{t}}w)^{{(\doublehat{b}')}^{\mathrm{t}}}
\]
we see that $\hat{q}(\theta,w, \varepsilon) = T_1(\theta,w,\varepsilon) - T_2(\theta,w,\varepsilon)$ and so $\hat{q}(\theta,w, \varepsilon) = 0$ if and only if $T_1(\theta,w, \varepsilon) = T_2(\theta,w, \varepsilon)$. Our goal is to compute
\[
T(\theta):=\hat{\beta}\left.D_w \hat{q}(\tilde{x} + \Gamma_0\theta,w, \varepsilon)\right|_{\varepsilon=0,w=z(\theta)^\gamma}
\]
Recall that $T_1(\theta,w, 0) = T_2(\theta,w,0)$ if and only if $w = z(\theta)^\gamma$. Differentiation gives
\begin{eqnarray*}
T(\theta)&=&\hat{\beta}\left[\mathrm{diag}(T_1(\theta,w,0))\,\,\hat{b}^{\mathrm{t}}\,\,\mathrm{diag}(\bm{1}/w) - \mathrm{diag}(T_2(\theta,w,0))\,\,{(\hat{b}')}^{\mathrm{t}}\,\,\mathrm{diag}(\bm{1}/w)\right]_{w=z(\theta)^\gamma}\\
& = & -\hat{\beta}\left[\mathrm{diag}(T_1(\theta,z(\theta)^\gamma,0))\,\,\hat{\beta}^{\mathrm{t}}\,\,\mathrm{diag}(\bm{1}/z(\theta)^\gamma)\right]\,.
\end{eqnarray*}
Thus whenever $z(\theta)$ is positive, both diagonal matrices in the last expression are well-defined and positive, and hence (as $\hat{\beta}$ is nonsingular) $T(\theta)$ is similar to a negative definite matrix, being the product of a negative definite matrix and a positive diagonal matrix. Consequently, it has eigenvalues which are real and negative.

The conditions of Lemma~\ref{lemsingpert} are satisfied, and thus there exists $\varepsilon_0 >0$ such that for any fixed $\varepsilon \in (0, \varepsilon_0)$, (\ref{eqlocal5}) undergoes the bifurcation $B$ unfolded transversely by the parameters $\kappa$, at a point $(\theta(\varepsilon), w(\varepsilon), \kappa(\varepsilon))$, where $\theta(\varepsilon)$, $w(\varepsilon)$ and $\kappa(\varepsilon)$  can be chosen to be $C^r$ for any finite $r$, and such that $\theta(0) = 0$, $w(0) = \tilde{x}^\gamma$ and $\kappa(0) = \tilde{\kappa}$. By reducing $\varepsilon_0$ if necessary, we can ensure that $\kappa(\varepsilon)$, $x(\varepsilon) = \tilde{x} + \Gamma_0\theta(\varepsilon)+\varepsilon\alpha\hat{\beta}^{-1} w(\varepsilon)$ and $y(\varepsilon) = (\varepsilon w(\varepsilon), \bm{1}-\varepsilon\delta^\mathrm{t}w(\varepsilon))$ are all positive, and so the bifurcation of (\ref{eqglobal5}) occurs on the positive orthant and at positive parameter values.

{\bf Enlargement E6.} The calculations for this enlargement follow \cite{banajisplit}. We enlarge $\mathcal{R}$ to create $\mathcal{R}'$ by splitting some reactions as follows. Assume that $\mathcal{R}$ includes the reactions $a_i \cdot \mathsf{X} \rightarrow b_i \cdot \mathsf{X}\,\, (i=1,\ldots,m')$ for some $m' \geq 1$. Let $k \geq 0$, and let $s_i$ and $\beta_i$ be arbitrary nonnegative vectors of lengths $n$ and $m'+k$ respectively, and let $\mathsf{Y}_1, \ldots, \mathsf{Y}_{m'+k}$ be a list of $m'+k$ new species denoted collectively by $\mathsf{Y}$. Let $\mathcal{R}'$ be a new CRN created from $\mathcal{R}$ by replacing each of the reactions $a_i \cdot \mathsf{X} \rightarrow b_i \cdot \mathsf{X}$ with a chain
\[
a_i \cdot \mathsf{X} \rightarrow s_i \cdot \mathsf{X} + \beta_i\cdot \mathsf{Y} \rightarrow b_i \cdot \mathsf{X},\,\,(i=1,\ldots,m')\,.
\]
By assumption, $\beta : = (\beta_1|\beta_2|\cdots|\beta_{m'})$ has rank $m'$. 

Define $s:=(s_1|s_2|\cdots|s_{m'})$. Define $\hat{\beta},\doublehat{\beta},\hat{y}$ and $\doublehat{y}$ as in the case of enlargement E5. We give the reactions with new reactant complexes mass action kinetics, and choose rate constants so that the reactions have rate vector $\bm{\varepsilon}^{-\hatt{\beta}^\mathrm{t}} \circ x^{s^\mathrm{t}} \circ y^{\beta^\mathrm{t}}$. Let the $i$th reaction to be split be reaction $j_i$, and define $\underline{v}(x,\kappa) := (v_{j_1}(x,\kappa), \ldots, v_{j_{m'}}(x,\kappa))^\mathrm{t}$ to be the vector of reaction rates associated with the reactions to be split. Define
\[
q(x,y,\kappa,\varepsilon) := \underline{v}(x,\kappa) - \bm{\varepsilon}^{-\hatt{\beta}^\mathrm{t}} \circ x^{s^\mathrm{t}} \circ y^{\beta^\mathrm{t}}\,.
\]
The assumptions so far give us that the dynamics of $\mathcal{R}'$ is governed by
\begin{equation}
\label{eqglobal6}
\left(\begin{array}{c}\dot x\\\dot y\end{array}\right) = \left(\begin{array}{cc}\Gamma&\alpha\\0&\beta\end{array}\right)\left(\begin{array}{c}v(x,\kappa)\\q(x,y,\kappa,\varepsilon)\end{array}\right)\,.
\end{equation}
Here we have defined $\alpha_i := s_i - b_i$, $(i = 1, \ldots, m')$ and $\alpha := (\alpha_1|\alpha_2|\cdots |\alpha_{m'})$. 

We define $\delta := -(\doublehat{\beta}\hatt{\beta}^{-1})^{\mathrm{t}}$ and can confirm that $\delta^{\mathrm{t}}\hatt{y}+\doublehat{y}$ is constant along trajectories of (\ref{eqglobal6}). We fix $\delta^{\mathrm{t}}\hatt{y}+\doublehat{y} = \mathbf{1}$. 

Define $\theta$ by $x = \tilde{x} + \Gamma_0\theta  + \alpha\hat{\beta}^{-1}\hat{y}$, and define $w = \varepsilon^{-1}\hat{y}$. Then $(\theta, w)$ can be regarded as a local coordinate on the positive stoichiometric class of $\mathcal{R}'$ which includes the point $(x,\hat{y},\doublehat{y}) = (\tilde{x},0,\bm{1})$ in its closure. Define $\hat{q}(\theta,w, \kappa, \varepsilon) = q(\tilde{x} + \Gamma_0\theta + \varepsilon\alpha\hat{\beta}^{-1}w,(\varepsilon w, \mathbf{1}-\varepsilon\delta^{\mathrm{t}}w),\kappa, \varepsilon)$ for $\varepsilon > 0$ sufficiently small. We can write out $\hat{q}(\theta,w,\kappa, \varepsilon)$ in full:
\[
\hat{q}(\theta,w,\kappa,\varepsilon) := \underline{v}(\tilde{x} + \Gamma_0\theta+\varepsilon\alpha\hatt{\beta}^{-1} w,\kappa) - (\tilde{x} + \Gamma_0\theta+\varepsilon\alpha\hatt{\beta}^{-1} w)^{s^\mathrm{t}} \circ w^{\hat{\beta}^\mathrm{t}}\circ (\mathbf{1}-\varepsilon\delta^{\mathrm{t}} w)^{\doublehat{\beta}^\mathrm{t}}\,,
\]
and note that $\hat{q}(\theta,w, \kappa, \varepsilon)$ has a smooth extension to $\varepsilon = 0$, with
\[
\hat{q}(\theta,w,\kappa,0) := \underline{v}(\tilde{x} + \Gamma_0\theta,\kappa) - (\tilde{x} + \Gamma_0\theta)^{s^\mathrm{t}} \circ w^{\hat{\beta}^\mathrm{t}}\,.
\]
We then have, for the evolution of $\theta$, $w$, and $\kappa$ 
\begin{equation}
\label{eqlocal6}
\begin{array}{rcl}\dot \theta & = & \Lambda  v(\tilde{x} + \Gamma_0\theta+\varepsilon\alpha\hatt{\beta}^{-1} w,\kappa)\,,\\
\varepsilon\dot {w} & = & \hatt{\beta} \hat{q}(\theta,w,\kappa,\varepsilon)\,,\\
\dot \kappa & = & 0\,.
\end{array}
\end{equation}
The equation $\hat{q}(\theta,w,\kappa,0) = 0$ can be solved to get 
\[
w = V(\tilde{x} + \Gamma_0\theta,\kappa)\circ (\tilde{x} + \Gamma_0\theta)^\gamma
\]
where $V(z,\kappa) = \underline{v}(z,\kappa)^{(\hatt{\beta}^{-1})^{\mathrm{t}}}$ and $\gamma = -(s\hatt{\beta}^{-1})^{\mathrm{t}}$. 

When we set $\varepsilon=0$ and $\hat{q}(\theta,w,\kappa,0)=0$, the dynamics of $\theta$ in equation (\ref{eqlocal6}) reduces to that of $\theta$ in (\ref{eqlocal}), and thus (\ref{eqlocal6}) is a singular perturbation of (\ref{eqlocal}) in the sense of Lemma~\ref{lemsingpert}.

By calculations to follow, adapted from \cite{banajisplit}, the eigenvalues of $D_{w}\hat{q}(\theta,w,\kappa,0)$ are real and negative on the zero set of $\hat{q}(\theta,w,\kappa,0)$ provided $\tilde{x} + \Gamma_0\theta$ is positive, which is true for sufficiently small $|\theta|$.

For ease of notation, define $z(\theta) := \tilde{x} + \Gamma_0\theta$, and $T(\theta) := \left.D_{w}\hat{q}(\theta,w,\kappa,\varepsilon)\right|_{\varepsilon=0,w=V(z(\theta),\kappa)\circ z(\theta)^\gamma}$. Then
\begin{eqnarray*}
T(\theta) &=& -\left.\hatt{\beta}\mathrm{diag}(w^{\hatt{\beta}^\mathrm{t}}\circ z(\theta)^{s^\mathrm{t}})\hatt{\beta}^\mathrm{t}\mathrm{diag}(\mathbf{1}/w)\right|_{w=V(z(\theta),\kappa)\circ z(\theta)^\gamma}\\
& = & -\hatt{\beta}\mathrm{diag}(\underline{v}(z(\theta),\kappa))\hatt{\beta}^\mathrm{t}\mathrm{diag}(\mathbf{1}/(V(z(\theta),\kappa)\circ z(\theta)^\gamma))\,.
\end{eqnarray*}
Provided $z(\theta)$ is positive, both diagonal matrices in the last expression are well-defined and positive, and hence (as $\hatt{\beta}$ is nonsingular) $T(\theta)$ is similar to a negative definite matrix (see \cite{banajisplit}). Consequently, it has eigenvalues which are real and negative. 

The conditions of Lemma~\ref{lemsingpert} are satisfied, and thus there exists $\varepsilon_0 >0$ such that for any fixed $\varepsilon \in (0, \varepsilon_0)$, (\ref{eqlocal6}) undergoes the bifurcation $B$ unfolded transversely by the parameters $\kappa$, at a point $(\theta(\varepsilon), w(\varepsilon), \kappa(\varepsilon))$, where $\theta(\varepsilon), w(\varepsilon)$ and $\kappa(\varepsilon)$ can be chosen to be $C^r$ for any finite $r$, and such that $\theta(0) = 0$, $w(0) = V(\tilde{x},\tilde{\kappa})\circ\tilde{x}^\gamma$ and $\kappa(0) = \tilde{\kappa}$. By reducing $\varepsilon_0$ if necessary, we can ensure that $\kappa(\varepsilon)$, $x(\varepsilon) = \tilde{x} + \Gamma_0\theta(\varepsilon)+\varepsilon\alpha\hat{\beta}^{-1} w(\varepsilon)$ and $y(\varepsilon) = (\varepsilon w(\varepsilon), \bm{1}-\varepsilon\delta^\mathrm{t}w(\varepsilon))$ are all positive, and so the bifurcation $B$ in (\ref{eqglobal6}) occurs on the positive orthant and at positive parameter values. 

This completes the proof of the theorem. \hfill
\end{myproof}

\begin{remark}[Relationships to other work]
  The content of Theorem~\ref{thmbifinherit} was foreshadowed in the conclusions of \cite{banajisplit}, and many of the details in the proof of Theorem~\ref{thmbifinherit} follow related results in \cite{banajiCRNosci,banajiCRNosci1,banajiboroshofbauer,banajisplit}. The proof of the inheritance of bifurcations result in the case of enlargement E1 was sketched in Remark~4.3 of \cite{banajiCRNosci}. The proof in the case of enlargement E3 was sketched in Remark~6 of \cite{banajiboroshofbauer}. The approach to enlargement E5 is simplified from the original approach in \cite{banajiCRNosci1}, and indeed implies a somewhat simpler proof of the main result of \cite{banajiCRNosci1}. Theorem~\ref{thmbifinherit} answers questions about bifurcations posed by Conradi and Shiu in \cite{conradishiubiophys}, and also implies some of the observations on Andronov--Hopf bifurcations in various models of a biologically important network in Obatake {\em et al.} \cite{obatake}. In \cite{ConradiMincheva2023}, Conradi and Mincheva use inheritance results to show that oscillations occur in a mass action system as a consequence of an Andronov--Hopf bifurcation in a subnetwork: the results here allow this claim to be strengthened, with an Andronov--Hopf bifurcation guaranteed, by Theorem~\ref{thmbifinherit}, in the larger system.
\end{remark}

\begin{remark}[Topological equivalence and normal forms]
  In the formulation and proof of Theorem~\ref{thmbifinherit} we make no assumptions about whether the enlarged systems admit topologically equivalent bifurcation diagrams to the original systems, even in the case of generic bifurcations. Indeed, as is well-known, for certain bifurcations such as the fold-Hopf bifurcation and the Hopf-Hopf bifurcation, the dynamics of the truncated normal form does not capture the dynamics of the normal form with generic higher order terms \cite[Sections 8.5~and~8.6]{kuznetsov:2023}. The only claim we make here is that those behaviours which can be inferred near the bifurcation point based on finitely many smooth conditions on Taylor coefficients of the original system (up to some arbitrary but finite order), survive in the enlarged system.
\end{remark}

When we restrict attention to fully open networks, we obtain an immediate, but potentially useful, corollary to Theorem~\ref{thmbifinherit}. Given CRNs $\mathcal{R}$ and $\mathcal{R}'$, we say that $\mathcal{R}$ is an {\em induced subnetwork} of $\mathcal{R}'$ \cite{banajipanteaMPNE} if $\mathcal{R}$ can be obtained from $\mathcal{R}'$ by removing some subset of the reactions of $\mathcal{R}'$, and/or deleting some species of $\mathcal{R}'$ from every reaction in which they occur.

\begin{cor}[Fully open CRNs and the induced subnetwork partial ordering]
\label{corfully}
Given fully open CRNs $\mathcal{R}$ and $\mathcal{R}'$, with $\mathcal{R}$ an induced subnetwork of $\mathcal{R}'$, if $\mathcal{R}$ admits some bifurcation with mass action kinetics, unfolded transversely by its rate constants, then so does $\mathcal{R}'$. 
\end{cor}
\begin{proof}
The assumptions that $\mathcal{R}$ and $\mathcal{R}'$ are fully open with $\mathcal{R}$ an induced subnetwork of $\mathcal{R}'$ imply that $\mathcal{R}'$ can be constructed from $\mathcal{R}$ via a sequence of enlargements of the form E1 and E4 (see the proof of Corollary~4.2 in \cite{banajipanteaMPNE}). The result is now immediate from Theorem~\ref{thmbifinherit}.
\end{proof}

\subsection{A worked example illustrating all six enlargements}
\label{secsimplefold}

The purpose of this section is to illustrate the terminology and calculations in the proof of Theorem~\ref{thmbifinherit} in a concrete setting, using the simplest bifurcation, namely a fold bifurcation. Examples of applying Theorem~\ref{thmbifinherit} to a variety of other bifurcations are presented in Section~\ref{secex}.

Consider the following rank $1$ mass action system on two chemical species $\mathsf{X}_1$ and $\mathsf{X}_2$:
\begin{equation}
\tag{$\mathcal{R}_0$}
\mathsf{X}_1+2\mathsf{X}_2\overset{1}{\longrightarrow} 3\mathsf{X}_2,\quad \mathsf{X}_2 \overset{\kappa}{\longrightarrow} \mathsf{X}_1\,\,.
\end{equation}
One rate constant has been fixed at $1$, while the other, $\kappa$, is allowed to vary. The network has been chosen because it displays a nondegenerate fold bifurcation as $\kappa$ is varied. We can use this network to illustrate the six enlargements, each giving rise to a larger network which must, by Theorem~\ref{thmbifinherit}, display a nondegenerate fold bifurcation. Denoting the concentration of $\mathsf{X}_i$ by $x_i$ ($i=1,2$), $\mathcal{R}_0$ gives rise to the system of ODEs
\[
\left(\begin{array}{c}
\dot x_1\\
\dot x_2
\end{array}
\right) 
=
\left(\begin{array}{c}
-x_1x_2^2 + \kappa x_2\\
x_1x_2^2 -\kappa x_2\,
\end{array}\right)
\,\, = \,\,\left(\begin{array}{rr}-1&1\\1&-1\end{array}\right)\left(\begin{array}{c}x_1x_2^2\\\kappa x_2\end{array}\right)\,.
\]
Let us focus on the positive stoichiometric class $\mathcal{S}_0=\{(x_1,x_2) \in \mathbb{R}^2_+\colon x_1+x_2=2\}$. We define a local coordinate $\theta$ on $\mathcal{S}_0$ via 
\[
\left(\begin{array}{c}x_1\\x_2\end{array}\right) = \left(\begin{array}{c}1\\1\end{array}\right) + \left(\begin{array}{r}-1\\1\end{array}\right)\theta\,.
\]
Here $(1,1)^\mathrm{t}$ corresponds to the initial point $\tilde{x}$ in the proof of Theorem~\ref{thmbifinherit}, and $(-1,1)^\mathrm{t}$ is the matrix $\Gamma_0$. The coordinate $\theta$ evolves according to
\[
\dot \theta = (1+\theta)(1-\kappa-\theta^2) =: f(\theta, \kappa)\,.
\]
Later, we may choose to augment this equation with the trivial equation $\dot \kappa = 0$. The domain of $\theta$, denoted $V_\theta$, may be taken to be the interval $(-1,1)$, ensuring that $(x_1, x_2)^\mathrm{t}$ remains positive, and the domain of $\kappa$, denoted $V_\kappa$, is $\mathbb{R}_+$. The system has a nondegenerate fold bifurcation at $(\theta,\kappa)=(0,1)$, and as $\kappa$ decreases through $1$, a pair of equilibria are born close to $\theta = 0$. We may consider the bifurcation to be defined in a neighbourhood of $(0,1)$ via the conditions: $f(\theta,\kappa)=0$, $f_\theta(\theta,\kappa)=0$, and the nondegeneracy and transversality conditions $f_{\theta\theta}(\theta,\kappa)<0$ and $f_{\kappa}(\theta,\kappa)<0$. Here, the condition $f_\theta = 0$ corresponds to the main bifurcation condition B2 in Section~\ref{secbif}. The conditions $f=0,f_\theta=0,f_{\theta\theta}<0$ define a submanifold in the space of $2$-jets $J^2(V_x, \mathbb{R}) \cong V_x \times \mathbb{R}^3$ corresponding to the fold bifurcation. The transversality condition B3 is that $(\theta,\kappa)$ must be a regular point of $(f,f_\theta)\colon V_\theta \times V_\kappa \to \mathbb{R}^2$. However, given $f_\theta(\theta,\kappa)= 0$ and $f_{\theta\theta}(\theta,\kappa)\neq 0$, the condition that $f_{\kappa}(\theta,\kappa)\neq 0$ is clearly equivalent to regularity of $(f,f_\theta)$ at $(\theta,\kappa)$, and consequently $f_{\kappa}(\theta,\kappa) < 0$ may be taken as the transversality condition in a neighbourhood of $(0,1)$.

{\bf Enlargement 1: a linearly dependent reaction.} Let us add to $\mathcal{R}_0$ the linearly dependent reaction $\mathsf{X}_1+\mathsf{X}_2 \rightarrow 2\mathsf{X}_2$ with rate constant $\varepsilon$ to obtain
\begin{equation}
\tag{$\mathcal{R}_1$}
\mathsf{X}_1+2\mathsf{X}_2\overset{1}{\longrightarrow} 3\mathsf{X}_2,\quad \mathsf{X}_2 \overset{\kappa}{\longrightarrow} \mathsf{X}_1\,,\quad \mathsf{X}_1+\mathsf{X}_2 \overset{\varepsilon}{\longrightarrow} 2\mathsf{X}_2\,.
\end{equation}
The enlarged system evolves according to
\[
\left(\begin{array}{c}
\dot x_1\\
\dot x_2
\end{array}
\right) 
\,\, = \,\,\left(\begin{array}{rrr}-1&1&-1\\1&-1&1\end{array}\right)\left(\begin{array}{c}x_1x_2^2\\\kappa x_2\\\varepsilon x_1x_2\end{array}\right)\,.
\]
The variable $\theta$ defined as above, remains a local coordinate on $\mathcal{S}_0$ (which is still a positive stoichiometric class of the enlarged system), and now evolves according to:
\[
\dot \theta = (1+\theta)(1-\kappa-\theta^2) +\varepsilon (1-\theta^2) =: \hat{f}(\theta, \kappa, \varepsilon)\,.
\]
The function $\hat{f}$ is smooth and $\hat{f}(\theta, \kappa, 0) = f(\theta, \kappa)$. According to Theorem~\ref{thmbifinherit}, for any sufficiently small $\varepsilon > 0$, there exists $(\theta(\varepsilon), \kappa(\varepsilon)) \in V_\theta \times V_\kappa$, close to $(0,1)$, and such that $\mathcal{R}_1$ has a fold bifurcation, unfolded transversely by $\kappa$, at $(x_1(\varepsilon), x_2(\varepsilon), \kappa(\varepsilon))$, where $x_1(\varepsilon) = 1-\theta(\varepsilon)$, $x_2(\varepsilon) = 1+ \theta(\varepsilon)$. 

{\bf Enlargement 2: the fully open extension.} Let us now modify $\mathcal{R}_0$ by including inflows and outflows for all the species involved, to get:
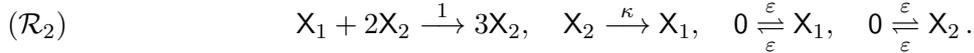
\begin{equation}
\tag{$\mathcal{R}_2$}
\mathsf{X}_1+2\mathsf{X}_2\overset{1}{\longrightarrow} 3\mathsf{X}_2,\quad \mathsf{X}_2 \overset{\kappa}{\longrightarrow} \mathsf{X}_1 , \quad \mathsf{0}\overset{\varepsilon}{\underset{\varepsilon}{\rightleftharpoons}} \mathsf{X}_1, \quad \mathsf{0}\overset{\varepsilon}{\underset{\varepsilon}{\rightleftharpoons}} \mathsf{X}_2\,.
\end{equation}
The rate constants of the flow reactions have been chosen to all be equal to $\varepsilon$. The enlarged network gives rise to the ODE
\[
\left(\begin{array}{c}
\dot x_1 \\
\dot x_2
\end{array}
\right) 
=
\left(\begin{array}{rrrr}-1&1&1&0\\1&-1&0&1\end{array}\right)\left(\begin{array}{c}x_1x_2^2\\\kappa x_2\\\varepsilon(1-x_1)\\\varepsilon(1-x_2)\end{array}\right)\,.
\]
We can confirm that $\mathcal{S}_0 := \{(x_1,x_2) \in \mathbb{R}^2_{+}\colon x_1+x_2=2\}$ remains locally invariant for any choice of $\varepsilon > 0$, although it is no longer a positive stoichiometric class for $\mathcal{R}_2$ as the unique positive stoichiometric class is now $\mathbb{R}^2_{+}$. Defining $\theta$ as a local coordinate on $\mathcal{S}_0$ as before we have, for the dynamics on $\mathcal{S}_0$,
\[
\dot \theta = (1+\theta)(1-\kappa-\theta^2) -\varepsilon \theta := \hat{f}(\theta, \kappa, \varepsilon)\,.
\]
The function $\hat{f}$ is smooth, and $\hat{f}(\theta, \kappa, 0) = f(\theta, \kappa)$. According to Theorem~\ref{thmbifinherit}, for any sufficiently small $\varepsilon > 0$, there exists $(\theta(\varepsilon), \kappa(\varepsilon)) \in V_\theta \times V_\kappa$, close to $(0,1)$, and such that restricting attention to $\mathcal{S}_0$, $\mathcal{R}_2$ has a fold bifurcation, unfolded transversely by $\kappa$, at $(x_1(\varepsilon), x_2(\varepsilon), \kappa(\varepsilon))$, where $x_1(\varepsilon) = 1-\theta(\varepsilon)$, $x_2(\varepsilon) = 1+ \theta(\varepsilon)$. Moreover, the eigenvalue at $(x_1(\varepsilon), x_2(\varepsilon))$ associated with directions transverse to $\mathcal{S}_0$ is just $-\varepsilon$.

{\bf Enlargement 3: a linearly dependent species.} Let us now modify $\mathcal{R}_0$ by including a new species in some reactions of $\mathcal{R}_0$ without altering the rank of the network, to obtain
\begin{equation}
\tag{$\mathcal{R}_3$}
\mathsf{X}_1+2\mathsf{X}_2\overset{1}{\longrightarrow} 3\mathsf{X}_2+\mathsf{Y},\quad \mathsf{X}_2 +\mathsf{Y} \overset{\varepsilon\kappa}{\longrightarrow} \mathsf{X}_1\,.
\end{equation}
The enlarged network gives rise to the ODE
\[
\left(\begin{array}{c}
\dot x_1 \\
\dot x_2\\
\dot y
\end{array}
\right) 
=
\left(\begin{array}{rr}-1&1\\1&-1\\1&-1\end{array}\right)\left(\begin{array}{c}x_1x_2^2\\\varepsilon\kappa x_2y\end{array}\right)\,.
\]
Note that $y-x_2$ is constant along trajectories, and we set the value of this constant to be $\varepsilon^{-1}$. In other words, we focus attention on the positive stoichiometric class 
\[
\mathcal{S}_\varepsilon:=\{(x_1,x_2,y) \in \mathbb{R}^3_+\,:\, x_1+x_2=2, \,\,y-x_2 = \varepsilon^{-1}\}\,
\]
of $\mathcal{R}_3$. A local coordinate $\theta$ on $\mathcal{S}_\varepsilon$, defined in the usual way, evolves according to
\[
\dot \theta = (1+\theta)(1-\kappa-\theta^2) - \varepsilon \kappa (1+\theta)^2  =: \hat{f}(\theta, \kappa, \varepsilon)\,.
\]
Note that $\hat{f}$ is smooth and $\hat{f}(\theta, \kappa, 0) = f(\theta, \kappa)$. According to Theorem~\ref{thmbifinherit}, for any sufficiently small $\varepsilon > 0$, there exists $(\theta(\varepsilon), \kappa(\varepsilon)) \in V_\theta \times V_\kappa$, close to $(0,1)$, and such that $\mathcal{R}_3$ has a fold bifurcation, unfolded transversely by $\kappa$, on the positive stoichiometric class $\mathcal{S}_\varepsilon$ at the point $(x_1(\varepsilon), x_2(\varepsilon), y(\varepsilon), \kappa(\varepsilon))$, where $x_1(\varepsilon) = 1-\theta(\varepsilon)$, $x_2(\varepsilon) = 1+ \theta(\varepsilon)$ and $y(\varepsilon) = 1+ \theta(\varepsilon) + \varepsilon^{-1}$. 

{\bf Enlargement 4: a new species with inflows and outflows.} Let us now modify $\mathcal{R}_0$ by including a new species $\mathsf{Y}$ in one of its reactions, and also including inflows and outflows for $\mathsf{Y}$, to obtain
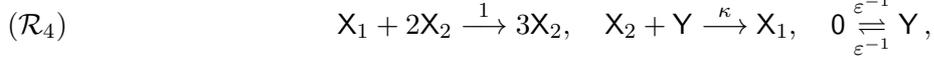
\begin{equation}
\tag{$\mathcal{R}_4$}
\mathsf{X}_1+2\mathsf{X}_2\overset{1}{\longrightarrow} 3\mathsf{X}_2,\quad \mathsf{X}_2 +\mathsf{Y} \overset{\kappa}{\longrightarrow} \mathsf{X}_1 , \quad \mathsf{0}\overset{\varepsilon^{-1}}{\underset{\varepsilon^{-1}}{\rightleftharpoons}} \mathsf{Y}\,,
\end{equation}
giving rise to the singularly perturbed ODE
\begin{equation}
\label{eqinout0}
\left(\begin{array}{c}
\dot x_1 \\
\dot x_2 \\
\varepsilon \dot y
\end{array}
\right) 
=
\left(\begin{array}{rrr}-1&1&0\\1&-1&0\\0&-\varepsilon&1\end{array}\right)\left(\begin{array}{c}x_1x_2^2\\\kappa x_2y\\1-y\end{array}\right)\,.
\end{equation}
We focus attention on the 2D, positive stoichiometric class of $\mathcal{R}_4$ defined by
\[
\mathcal{S}' := \{(x_1,x_2,y) \in \mathbb{R}^3_{+}\colon x_1+x_2 = 2\}\,.
\]
In local coordinates $(\theta, y)$ on $\mathcal{S}'$, we have the system
\begin{equation}
\label{eqinout}
\begin{array}{rcl}
\dot \theta &=& (1+\theta)(1-\kappa y-\theta^2)\,,\\
\varepsilon \dot y & = & -\varepsilon \kappa (1+\theta) y + (1-y)\,,\\
\dot \kappa & = & 0\,.\end{array}
\end{equation}
For sufficiently small $\varepsilon > 0$, the system has a locally invariant manifold $\mathcal{E}_\varepsilon$ close, on compact sets, to $\mathcal{E}_0 = \{(\theta,y,\kappa)\,\colon\, y=1\}$. In particular, there exist $\varepsilon_1 > 0$, open neighbourhoods $U_\theta \subseteq V_\theta$, $U_\kappa \subseteq \mathbb{R}_+$ with $0 \in U_\theta$, $1 \in U_\kappa$, and a function $\psi\colon U_\theta \times U_\kappa \times (-\varepsilon_1,\varepsilon_1) \to \mathbb{R}$, of any desired finite level of differentiability, satisfying $\psi(\theta,\kappa,0) = 1$, and such that $\mathcal{E}_\varepsilon:=\{(\theta,y,\kappa)\colon (\theta,\kappa) \in U_\theta \times U_\kappa, y = \psi(\theta,\kappa,\varepsilon)\}$ is, for $\varepsilon \in (-\varepsilon_1, \varepsilon_1)\backslash\{0\}$, a locally invariant manifold of (\ref{eqinout}). Restricting attention to $\mathcal{E}_\varepsilon$, we get
\begin{equation}
\label{eqinout1}
\dot \theta = (1+\theta)(1-\kappa \psi(\theta,\kappa,\varepsilon)-\theta^2) =: \hat{f}(\theta, \kappa, \varepsilon)\,.
\end{equation}
Observe that $\hat{f}$ is as differentiable as $\psi$, and that $\hat{f}(\theta, \kappa, 0) = f(\theta, \kappa)$. According to the proof of Theorem~\ref{thmbifinherit}, there exists $\varepsilon_0 \in (0, \varepsilon_1]$ such that whenever $\varepsilon \in (0,\varepsilon_0)$, there exists $(\theta(\varepsilon), \kappa(\varepsilon)) \in V_\theta \times V_\kappa$, close to $(0,1)$, and such that, restricting attention to $\mathcal{E}_\varepsilon$, (\ref{eqinout1}) has a fold bifurcation, unfolded transversely by $\kappa$ at $(\theta(\varepsilon),\kappa(\varepsilon))$. Moreover the eigenvalue of (\ref{eqinout}) at the bifurcating equilibrium transverse to $\mathcal{E}_\varepsilon$ is real and negative. 

Finally, we observe that, reducing $\varepsilon_0$ if necessary, for each $\varepsilon \in (0, \varepsilon_0)$, $\kappa(\varepsilon)$, $x_1(\varepsilon) = 1-\theta(\varepsilon)$, $x_2(\varepsilon) = 1+ \theta(\varepsilon)$ and $y(\varepsilon) = \psi(\theta(\varepsilon),\kappa(\varepsilon),\varepsilon)$ are all positive, and so the bifurcation takes place at a positive parameter value, and the bifurcating equilibrium is a positive equilibrium of (\ref{eqinout0}).

{\bf Enlargement 5: new reversible reactions involving new species.} Let us now add to $\mathcal{R}_0$ a new reversible reaction $\mathsf{Y}_1+\mathsf{X}_1\rightleftharpoons 2\mathsf{Y}_2$, involving two new species $\mathsf{Y}_1$ and $\mathsf{Y}_2$ to obtain:
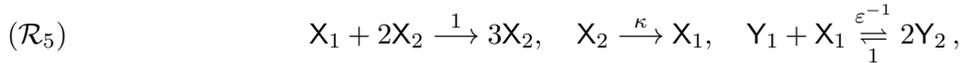
\begin{equation}
\tag{$\mathcal{R}_5$}
\mathsf{X}_1+2\mathsf{X}_2\overset{1}{\longrightarrow} 3\mathsf{X}_2,\quad \mathsf{X}_2 \overset{\kappa}{\longrightarrow} \mathsf{X}_1 , \quad \mathsf{Y}_1+\mathsf{X}_1 \overset{\varepsilon^{-1}}{\underset{1}{\rightleftharpoons}} 2\mathsf{Y}_2\,,
\end{equation}
giving rise to the ODE
\begin{equation}
\label{eqrev0}
\left(\begin{array}{c}
\dot x_1 \\
\dot x_2 \\
\dot y_1 \\
\dot y_2
\end{array}
\right) 
=
\left(\begin{array}{rrr}-1&1&-1\\1&-1&0\\0&0&-1\\0&0&2\end{array}\right)\left(\begin{array}{c}x_1x_2^2\\\kappa x_2\\\varepsilon^{-1}y_1x_1 -y_2^2\end{array}\right)\,.
\end{equation}
We focus attention on the 2D, positive stoichiometric class of $\mathcal{R}_5$ defined by
\[
\mathcal{S}' = \{(x_1,x_2,y_1,y_2) \in \mathbb{R}^4_{+}\colon x_1+x_2-y_1 = 2,\, 2y_1+y_2=1\}\,.
\]
Define $\theta$ by 
\[
\left(\begin{array}{c}x_1-y_1\\x_2\end{array}\right) = \left(\begin{array}{c}1-\theta\\1+\theta\end{array}\right)\,,
\]
and define $w = \varepsilon^{-1}y_1$. For any fixed $\varepsilon > 0$,  $(\theta, w)$ is a local coordinate on $\mathcal{S}'$. We get, in these coordinates,
\begin{equation}
\label{eqrev}
\begin{array}{rcl}
\dot \theta & = & (1+\theta)(1-\kappa-\theta^2 + \varepsilon w(1+\theta))\,,\\
\varepsilon \dot w & = & 1-w(1-\theta)-\varepsilon (w^2+4w - 4\varepsilon w^2)\,,\\
\dot \kappa & = & 0\,.\end{array}
\end{equation}
For sufficiently small $\varepsilon > 0$, the system has a locally invariant manifold $\mathcal{E}_\varepsilon$ close to 
\[
\mathcal{E}_0 = \{(\theta,w,\kappa)\,\colon\, w=\frac{1}{1-\theta}\}\,.
\] 
In particular, there exist $\varepsilon_1 > 0$, open neighbourhoods $U_\theta \subseteq V_\theta$, $U_\kappa \subseteq \mathbb{R}_+$ with $0 \in U_\theta$, $1 \in U_\kappa$, and a function $\psi\colon U_\theta \times U_\kappa \times (-\varepsilon_1,\varepsilon_1) \to \mathbb{R}$, of any desired finite level of differentiability and satisfying $\psi(\theta,\kappa,0) = \frac{1}{1-\theta}$, and such that $\mathcal{E}_\varepsilon:=\{(\theta,w,\kappa)\colon (\theta,\kappa) \in U_\theta \times U_\kappa, w = \psi(\theta,\kappa,\varepsilon)\}$ is, for $\varepsilon \in (-\varepsilon_1, \varepsilon_1)\backslash\{0\}$, a locally invariant manifold of (\ref{eqrev}). Restricting attention to $\mathcal{E}_\varepsilon$, we get
\begin{equation}
\label{eqrev1}
\dot \theta = (1+\theta)(1-\kappa-\theta^2 + \varepsilon(1+\theta)\psi(\theta,\kappa,\varepsilon)):=\hat{f}(\theta,\kappa,\varepsilon)\,.
\end{equation}
Observe that $\hat{f}$ is as differentiable as $\psi$, and that $\hat{f}(\theta, \kappa, 0) = f(\theta, \kappa)$. According to the proof of Theorem~\ref{thmbifinherit}, there exists $\varepsilon_0 \in (0, \varepsilon_1]$ such that whenever $\varepsilon \in (0,\varepsilon_0)$ there exists $(\theta(\varepsilon), \kappa(\varepsilon)) \in V_\theta \times V_\kappa$, close to $(0,1)$, and such that, restricting attention to $\mathcal{E}_\varepsilon$, (\ref{eqrev1}) has a fold bifurcation, unfolded transversely by $\kappa$ at $(\theta(\varepsilon),\kappa(\varepsilon))$. Moreover the eigenvalue of (\ref{eqrev}) at the bifurcating equilibrium transverse to $\mathcal{E}_\varepsilon$ is real and negative. 

Finally, we observe that, reducing $\varepsilon_0$ if necessary, if $\varepsilon \in (0, \varepsilon_0)$, then $\kappa(\varepsilon)$, $y_1(\varepsilon) = \varepsilon\psi(\theta(\varepsilon),\kappa(\varepsilon),\varepsilon)$, $y_2(\varepsilon) = 1-2y_1(\varepsilon)$, $x_1(\varepsilon) = 1+y_1(\varepsilon)-\theta(\varepsilon)$, and $x_2(\varepsilon) = 1+ \theta(\varepsilon)$ are all positive, and so the bifurcation takes place at a positive parameter value, and the bifurcating equilibrium is a positive equilibrium of (\ref{eqrev0}).

{\bf Enlargement 6: splitting reactions.} Finally, let us ``split'' the second reaction of $\mathcal{R}_0$ and insert a new intermediate complex involving two new species $\mathsf{Y}_1$ and $\mathsf{Y}_2$ to obtain:
\begin{equation}
\tag{$\mathcal{R}_6$}
\mathsf{X}_1+2\mathsf{X}_2\overset{1}{\longrightarrow} 3\mathsf{X}_2,\quad \mathsf{X}_2 \overset{\kappa}{\longrightarrow} \mathsf{Y}_1+\mathsf{Y}_2 \overset{\varepsilon^{-1}}{\longrightarrow} \mathsf{X}_1\,.
\end{equation}
This gives rise to the ODE
\begin{equation}
\label{eqsplit0}
\left(\begin{array}{c}
\dot x_1 \\
\dot x_2 \\
\dot y_1 \\
\dot y_2
\end{array}
\right) 
=
\left(\begin{array}{rrr}-1&1&-1\\1&-1&0\\0&0&1\\0&0&1\end{array}\right)\left(\begin{array}{c}x_1x_2^2\\\kappa x_2\\\kappa x_2 - \varepsilon^{-1}y_1y_2\end{array}\right)\,.
\end{equation}
We focus attention on the 2D, positive, stoichiometric class of $\mathcal{R}_6$, 
\[
\mathcal{S}' = \{(x_1,x_2,y_1,y_2) \in \mathbb{R}^4_{+}\colon x_1+x_2+y_1 = 2,\, y_2-y_1=1\}\,.
\]
Define $\theta$ by 
\[
\left(\begin{array}{c}x_1+y_1\\x_2\end{array}\right) = \left(\begin{array}{c}1-\theta\\1+\theta\end{array}\right)\,,
\]
and define $w = \varepsilon^{-1}y_1$. For any fixed $\varepsilon > 0$,  $(\theta, w)$ is a local coordinate on $\mathcal{S}'$. We get, in these coordinates,
\begin{equation}
\label{eqsplit}
\begin{array}{rcl}
\dot \theta & = & (1+\theta)(1-\kappa-\theta^2-\varepsilon w(1+\theta))\,,\\
\varepsilon \dot w & = & \kappa (1+\theta) - w - \varepsilon w^2\,,\\
\dot \kappa & = & 0\,.
\end{array}
\end{equation}
For sufficiently small $\varepsilon > 0$, the system has a locally invariant manifold $\mathcal{E}_\varepsilon$ close to 
\[
\mathcal{E}_0 = \{(\theta,w,\kappa)\,\colon\, w=\frac{1}{\kappa(1+\theta)}\}\,.
\] 
In particular, there exist $\varepsilon_1 > 0$, open neighbourhoods $U_\theta \subseteq V_\theta$, $U_\kappa \subseteq \mathbb{R}_+$ with $0 \in U_\theta$, $1 \in U_\kappa$, and a function $\psi\colon U_\theta \times U_\kappa \times (-\varepsilon_1,\varepsilon_1) \to \mathbb{R}$, of any desired finite level of differentiability and satisfying $\psi(\theta,\kappa,0) = \frac{1}{\kappa(1+\theta)}$, and such that $\mathcal{E}_\varepsilon:=\{(\theta,w,\kappa)\colon (\theta,\kappa) \in U_\theta \times U_\kappa, w = \psi(\theta,\kappa,\varepsilon)\}$ is, for $\varepsilon \in (-\varepsilon_1, \varepsilon_1)\backslash\{0\}$, a locally invariant manifold of (\ref{eqsplit}). Restricting attention to $\mathcal{E}_\varepsilon$, we get
\begin{equation}
\label{eqsplit1}
\dot \theta = (1+\theta)(1-\kappa-\theta^2-\varepsilon (1+\theta)\psi(\theta,\kappa,\varepsilon)):=\hat{f}(\theta,\kappa,\varepsilon)\,.
\end{equation}
Observe that $\hat{f}$ is as differentiable as $\psi$, and that $\hat{f}(\theta, \kappa, 0) = f(\theta, \kappa)$. According to the proof of Theorem~\ref{thmbifinherit}, there exists $\varepsilon_0 \in (0, \varepsilon_1]$ such that whenever $\varepsilon \in (0, \varepsilon_0)$, there exists $(\theta(\varepsilon), \kappa(\varepsilon)) \in V_\theta \times V_\kappa$, close to $(0,1)$, and such that, restricting attention to $\mathcal{E}_\varepsilon$, (\ref{eqsplit1}) has a fold bifurcation, unfolded transversely by $\kappa$ at $(\theta(\varepsilon),\kappa(\varepsilon))$. Moreover the eigenvalue of (\ref{eqsplit}) at the bifurcating equilibrium transverse to $\mathcal{E}_\varepsilon$ is real and negative. 

Finally, we observe that, reducing $\varepsilon_0$ if necessary, if $\varepsilon \in (0, \varepsilon_0)$, then $\kappa(\varepsilon)$, $y_1(\varepsilon) = \varepsilon\psi(\theta(\varepsilon),\kappa(\varepsilon),\varepsilon)$, $y_2(\varepsilon) = 1+y_1(\varepsilon)$, $x_1(\varepsilon) = 1-y_1(\varepsilon)-\theta(\varepsilon)$ and $x_2(\varepsilon) = 1+ \theta(\varepsilon)$ are all positive, and so the bifurcation takes place at a positive parameter value, and the bifurcating equilibrium is a positive equilibrium of (\ref{eqsplit0}).

\section{Some examples}
\label{secex}

We present three examples which illustrate the results in this paper. In the first two examples, we present networks which can be proven to display various bifurcations of codimension one or two using Theorem~\ref{thmbifinherit}. In the final example, we illustrate that Theorem~\ref{thmbifinherit} can be applied to networks with degenerate bifurcations, provided these are still unfolded transversely by the rate constants of the network; in this case, however, Theorem~\ref{thmbifinherit} cannot tell us about the degeneracy or otherwise of the bifurcation in the enlarged network.

\subsection{A supercritical Bautin bifurcation}

Consider the following three CRNs:

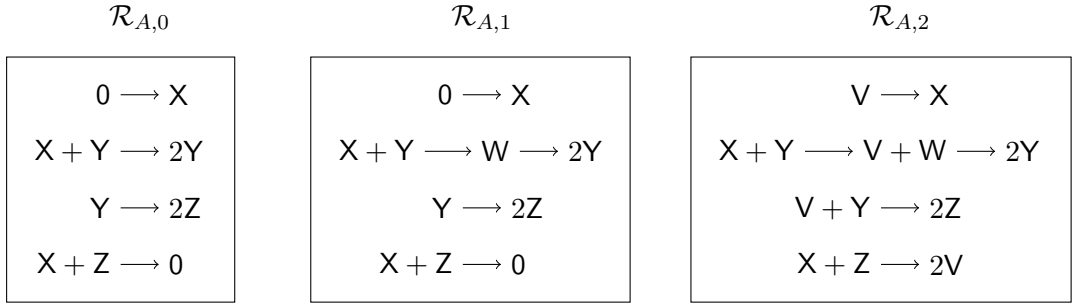
\begin{center}
\begin{tikzpicture}
\node at (0.25,1) {$\mathcal{R}_{A,0}$};
\draw (-1.5,0.5) rectangle (1.5,-2.75);
    \node[left]   (0) at (0  ,0   ) {$\mathsf{0}$};
    \node[right]  (X) at (0.5,0   ) {$\mathsf{X}$};
    \node[left] (XY) at (0  ,-0.75) {$\mathsf{X}+\mathsf{Y}$};
    \node[right] (2Y) at (0.5,-0.75) {$2\mathsf{Y}$};
    \node[left] (Y) at (0  ,-1.5) {$\mathsf{Y}$};
    \node[right] (2Z) at (0.5,-1.5) {$2\mathsf{Z}$};
    \node[left] (XZ) at (0  ,-2.25) {$\mathsf{X}+\mathsf{Z}$};
    \node[right] (0A) at (0.5,-2.25) {$\mathsf{0}$};
    \draw[->]   (0) to node {} (X);
    \draw[->] (XY) to node {} (2Y);
    \draw[->] (Y) to node {} (2Z);
    \draw[->] (XZ) to node {} (0A);

\begin{scope}[xshift=4.5cm]
\node at (0.25,1) {$\mathcal{R}_{A,1}$};
\draw (-2,0.5) rectangle (2.2,-2.75);
    \node[left]   (0) at (0  ,0   ) {$\mathsf{0}$};
    \node[right]  (X) at (0.5,0   ) {$\mathsf{X}$};
    \node[left] (XY) at (-0.5 ,-0.75) {$\mathsf{X}+\mathsf{Y}$};
    \node[left] (W) at (0.75,-0.75) {$\mathsf{W}$};
    \node[right] (2Y) at (1.25,-0.75) {$2\mathsf{Y}$};
    \node[left] (Y) at (0  ,-1.5) {$\mathsf{Y}$};
    \node[right] (2Z) at (0.5,-1.5) {$2\mathsf{Z}$};
    \node[left] (XZ) at (0  ,-2.25) {$\mathsf{X}+\mathsf{Z}$};
    \node[right] (0A) at (0.5,-2.25) {$\mathsf{0}$};
    \draw[->]   (0) to node {} (X);
    \draw[->] (XY) to node {} (W);
    \draw[->] (W) to node {} (2Y);
    \draw[->] (Y) to node {} (2Z);
    \draw[->] (XZ) to node {} (0A);
\end{scope}

\begin{scope}[xshift=10cm]
\draw (-2.5,0.5) rectangle (2.5,-2.75);
\node at (0.25,1) {$\mathcal{R}_{A,2}$};
    \node[left]   (V) at (0  ,0   ) {$\mathsf{V}$};
    \node[right]  (X) at (0.5,0   ) {$\mathsf{X}$};
    \node[left] (XY) at (-1  ,-0.75) {$\mathsf{X}+\mathsf{Y}$};
    \node[left] (VW) at (1,-0.75) {$\mathsf{V}+\mathsf{W}$};
    \node[right] (2Y) at (1.5,-0.75) {$2\mathsf{Y}$};
    \node[left] (VY) at (0  ,-1.5) {$\mathsf{V}+\mathsf{Y}$};
    \node[right] (2Z) at (0.5,-1.5) {$2\mathsf{Z}$};
    \node[left] (XZ) at (0  ,-2.25) {$\mathsf{X}+\mathsf{Z}$};
    \node[right] (2V) at (0.5,-2.25) {$2\mathsf{V}$};
    \draw[->]   (V) to node {} (X);
    \draw[->] (XY) to node {} (VW);
    \draw[->] (VW) to node {} (2Y);
    \draw[->] (VY) to node {} (2Z);
    \draw[->] (XZ) to node {} (2V);
\end{scope}
\end{tikzpicture}
\end{center}
The largest of the three, $\mathcal{R}_{A,2}$, is a rank $4$, bimolecular CRN on $5$ chemical species $\mathsf{V}, \mathsf{W}, \mathsf{X}, \mathsf{Y}$ and $\mathsf{Z}$. Further, no species figures on both sides of any reaction; and the network is {\em homogeneous}, i.e., the total concentration of the five species is conserved. We claim that $\mathcal{R}_{A,2}$ admits, with mass action kinetics and on some positive stoichiometric class, a supercritical Bautin bifurcation \cite[Section 8.3]{kuznetsov:2023}. Consequently it admits several interesting behaviours including: both sub- and supercritical Andronov--Hopf bifurcations; the coexistence of a stable periodic orbit and a stable equilibrium; and a fold bifurcation of periodic orbits. Confirming these claims without appealing to Theorem~\ref{thmbifinherit} would be challenging.

To see why $\mathcal{R}_{A,2}$ admits a supercritical Bautin bifurcation, we observe, first, that it was shown in \cite{banajiborosnonlinearity} that $\mathcal{R}_{A,0}$, with mass action kinetics, admits a supercritical Bautin bifurcation, unfolded transversely by its rate constants. In fact, network $\mathcal{R}_{A,0}$ is a minimal example of a bimolecular mass action network admitting a Bautin bifurcation. 

Network $\mathcal{R}_{A,2}$ is built from $\mathcal{R}_{A,0}$ via two successive enlargments. First, we introduce a new intermediate species $\mathsf{W}$ into the second reaction of $\mathcal{R}_{A,0}$ to obtain $\mathcal{R}_{A,1}$; the reader may easily confirm that this is an instance of enlargement E6 above. Second, we introduce into some of the reactions of $\mathcal{R}_{A,1}$ a new species $\mathsf{V}$, to obtain $\mathcal{R}_{A,2}$. This has been done without altering the rank of the network, and the addition of $\mathsf{V}$ is thus an instance of enlargement E3 above. 

By Theorem~\ref{thmbifinherit}, we immediately have that $\mathcal{R}_{A,2}$ admits a supercritical Bautin bifurcation on some positive stoichiometric class, and this bifurcation is unfolded transversely by the rate constants. Moreover, the proofs are constructive: if we know the rate constants at which the bifurcation occurs in $\mathcal{R}_{A,0}$, then we know how to choose a positive stoichiometric class and rate constants at which $\mathcal{R}_{A,2}$ admits this bifurcation.

\subsection{The homogenised Brusselator: a multitude of bifurcations}

Consider the following three CRNs. 
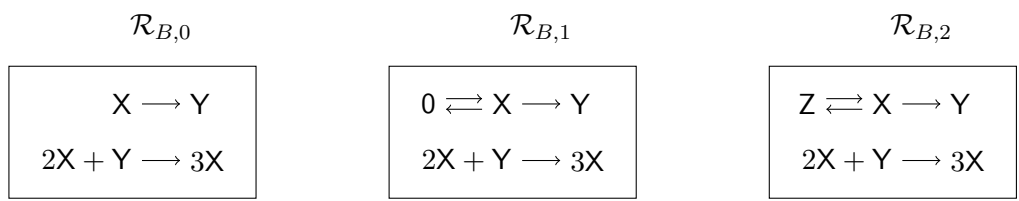
\begin{center}
\begin{tikzpicture}
\node at (0.25,1) {$\mathcal{R}_{B,0}$};
\draw (-1.75,0.5) rectangle (1.5,-1.25);
    \node[left]   (X) at (0  ,0   ) {$\mathsf{X}$};
    \node[right]  (Y) at (0.5,0   ) {$\mathsf{Y}$};
    \node[left] (2XY) at (0  ,-0.75) {$2\mathsf{X}+\mathsf{Y}$};
    \node[right] (3X) at (0.5,-0.75) {$3\mathsf{X}$};
    \draw[->]   (X) to node {} (Y);
    \draw[->] (2XY) to node {} (3X);

\begin{scope}[xshift=5cm]
\node at (0.25,1) {$\mathcal{R}_{B,1}$};
\draw (-1.75,0.5) rectangle (1.5,-1.25);
    \node[left] (0) at (-1,0) {$\mathsf{0}$};
    \node[left] (X) at (0,0) {$\mathsf{X}$};
    \node[right] (Y) at (0.5,0) {$\mathsf{Y}$};
    \node[left] (2XY) at (0,-0.75) {$2\mathsf{X}+\mathsf{Y}$};
    \node[right] (3X) at (0.5,-0.75) {$3\mathsf{X}$};
    \draw[->,transform canvas={yshift=2pt}]  (0) to node {} (X);
    \draw[->,transform canvas={yshift=-2pt}] (X) to node {} (0);
    \draw[->]   (X) to node {} (Y);
    \draw[->] (2XY) to node {} (3X);
\end{scope}

\begin{scope}[xshift=10cm]
\node at (0.25,1) {$\mathcal{R}_{B,2}$};
\draw (-1.75,0.5) rectangle (1.5,-1.25);
    \node[left] (Z) at (-1,0) {$\mathsf{Z}$};
    \node[left] (X) at (0,0) {$\mathsf{X}$};
    \node[right] (Y) at (0.5,0) {$\mathsf{Y}$};
    \node[left] (2XY) at (0,-0.75) {$2\mathsf{X}+\mathsf{Y}$};
    \node[right] (3X) at (0.5,-0.75) {$3\mathsf{X}$};
    \draw[->,transform canvas={yshift=2pt}]  (Z) to node {} (X);
    \draw[->,transform canvas={yshift=-2pt}] (X) to node {} (Z);
    \draw[->]   (X) to node {} (Y);
    \draw[->] (2XY) to node {} (3X);
\end{scope}
\end{tikzpicture}
\end{center}

As discussed in Section~\ref{secsimplefold}, the network $\mathcal{R}_{B,0}$ with mass action kinetics admits a nondegenerate fold bifurcation. On the other hand, it is well-known and easily checked that the \emph{Brusselator}, network $\mathcal{R}_{B,1}$, has a unique positive equilibrium for each choice of mass action rate constants; and this equilibrium undergoes a supercritical Andronov--Hopf bifurcation as the rate constants are varied. Notice that $\mathcal{R}_{B,1}$ cannot be obtained from $\mathcal{R}_{B,0}$ by applying a sequence of enlargements E1--E6. However, the \emph{homogenised Brusselator}, $\mathcal{R}_{B,2}$, introduced in \cite[Section 4.4]{banajiboroshofbauer}, is obtained from $\mathcal{R}_{B,0}$ by applying enlargement E5, and, at the same time, can also be seen as a result of applying enlargement E3 to $\mathcal{R}_{B,1}$. 

Hence, by Theorem~\ref{thmbifinherit}, $\mathcal{R}_{B,2}$ admits both a nondegenerate fold bifurcation, inherited from $\mathcal{R}_{B,0}$, and a supercritical Andronov--Hopf bifurcation, inherited from $\mathcal{R}_{B,1}$. In fact, as shown in \cite[Appendix A]{banajiboroshofbauer}, the homogenised Brusselator $\mathcal{R}_{B,2}$ even admits a supercritical Bogdanov--Takens bifurcation. This codimension-two bifurcation exhibits a fold bifurcation, an Andronov--Hopf bifurcation, and even a homoclinic bifurcation near the bifurcation point (see \cite[Section 8.4]{kuznetsov:2023}). Further, it was shown in \cite[Section 4.4]{banajiboroshofbauer} that $\mathcal{R}_{B,2}$ admits another codimension-two bifurcation: Bautin bifurcation (with a positive second Lyapunov coefficient). Hence, an attracting and a repelling limit cycle coexist on the parameter-dependent centre manifold for some choices of rate constants, and a fold bifurcation of limit cycles is also admitted. As a consequence, by Theorem~\ref{thmbifinherit}, any network that can be obtained by the application of a finite sequence of enlargements E1--E6 to $\mathcal{R}_{B,2}$ admits all the behaviours listed above. For example, the network
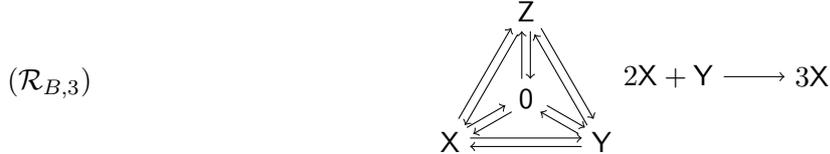
\begin{align}
\tag{$\mathcal{R}_{B,3}$}
\begin{aligned}
\begin{tikzpicture}[scale=2]
    \node (X) at (0,0) {$\mathsf{X}$};
    \node (Y) at (1,0) {$\mathsf{Y}$};
    \node (Z) at (0.5,0.87) {$\mathsf{Z}$};
    \node (0) at (0.5,0.29) {$\mathsf{0}$};
    \node[left] (2XY) at (1.8,0.435) {$2\mathsf{X}+\mathsf{Y}$};
    \node[right] (3X) at (2.2,0.435) {$3\mathsf{X}$};
    \draw[->,transform canvas={yshift=1.6pt}]                (X) to node {} (Y);
    \draw[->,transform canvas={yshift=-1.6pt}]               (Y) to node {} (X);
    \draw[->,transform canvas={xshift=-1.2pt,yshift=1.2pt}]  (X) to node {} (Z);
    \draw[->,transform canvas={xshift=1.2pt,yshift=-1.2pt}]  (Z) to node {} (X);
    \draw[->,transform canvas={xshift=-1.2pt,yshift=-1.2pt}] (Y) to node {} (Z);
    \draw[->,transform canvas={xshift=1.2pt,yshift=1.2pt}]   (Z) to node {} (Y);
    \draw[->,transform canvas={xshift=-1.6pt}]               (0) to node {} (Z);
    \draw[->,transform canvas={xshift=1.6pt}]                (Z) to node {} (0);
    \draw[->,transform canvas={xshift=1.2pt,yshift=-1.2pt}]  (0) to node {} (X);
    \draw[->,transform canvas={xshift=-1.2pt,yshift=1.2pt}]  (X) to node {} (0);
    \draw[->,transform canvas={xshift=1.2pt,yshift=1.2pt}]   (0) to node {} (Y);
    \draw[->,transform canvas={xshift=-1.2pt,yshift=-1.2pt}] (Y) to node {} (0);
    \draw[->] (2XY) to node {} (3X);
\end{tikzpicture}
\end{aligned}
\end{align}
which is the result of applying enlargements E1 and E2 to $\mathcal{R}_{B,2}$, admits the following bifurcations: fold (of equilibria), Andronov--Hopf, Bogdanov--Takens, homoclinic, Bautin, and fold (of limit cycles). Given the high number of parameters (i.e., rate constants) and that the network has rank three, confirming the occurrence of these bifurcations by direct calculations would be challenging.

As a final remark on this example, we note that as $\mathcal{R}_{B,3}$ is fully open, Corollary~\ref{corfully} allows us to conclude that any fully open mass action CRN which includes $\mathcal{R}_{B,3}$ as an induced subnetwork admits all of the above bifurcations too.

\subsection{Lifting a degenerate Andronov--Hopf bifurcation}
\label{secliftedvertical}

In this example we show that we may apply Theorem~\ref{thmbifinherit} to systems which display degenerate bifurcations. However, we should not expect the bifurcation in the enlarged network to be degenerate. We consider the following pair of mass action systems

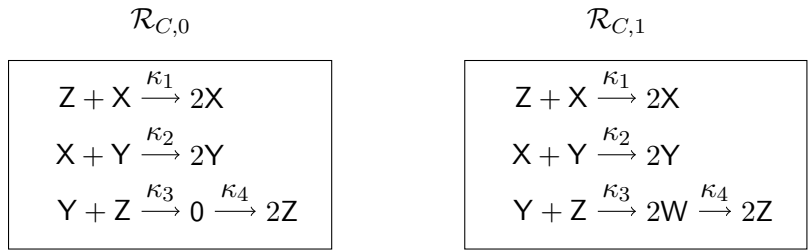
\begin{center}
\begin{tikzpicture}
\node at (0.25,1) {$\mathcal{R}_{C,0}$};
\draw (-1.75,0.5) rectangle (2.5,-2);

    \node[left]  (ZX) at (0  , 0  ) {$\mathsf{Z+X}$};
    \node[right] (2X) at (0.5, 0  ) {$2\mathsf{X}$};
    \node[left]  (XY) at (0  ,-0.75) {$\mathsf{X+Y}$};
    \node[right] (2Y) at (0.5,-0.75) {$2\mathsf{Y}$};
    \node[left]  (YZ) at (0  ,-1.5  ) {$\mathsf{Y+Z}$};
    \node[right]  (0) at (0.5,-1.5  ) {$\mathsf{0}$};
    \node[right] (2Z) at (1.5,-1.5  ) {$2\mathsf{Z}$};
    \draw[->]   (ZX) to node[above] {$\kappa_1$} (2X);
    \draw[->]   (XY) to node[above] {$\kappa_2$} (2Y);
    \draw[->]   (YZ) to node[above] {$\kappa_3$}  (0);
    \draw[->]    (0) to node[above] {$\kappa_4$} (2Z);

\begin{scope}[xshift=6cm]
\node at (0.25,1) {$\mathcal{R}_{C,1}$};
\draw (-1.75,0.5) rectangle (2.75,-2);

    \node[left]  (ZX) at (0  , 0  ) {$\mathsf{Z+X}$};
    \node[right] (2X) at (0.5, 0  ) {$2\mathsf{X}$};
    \node[left]  (XY) at (0  ,-0.75) {$\mathsf{X+Y}$};
    \node[right] (2Y) at (0.5,-0.75) {$2\mathsf{Y}$};
    \node[left]  (YZ) at (0  ,-1.5  ) {$\mathsf{Y+Z}$};
    \node[right] (2W) at (0.5,-1.5  ) {$2\mathsf{W}$};
    \node[right] (2Z) at (1.75,-1.5  ) {$2\mathsf{Z}$};
    \draw[->]   (ZX) to node[above] {$\kappa_1$} (2X);
    \draw[->]   (XY) to node[above] {$\kappa_2$} (2Y);
    \draw[->]   (YZ) to node[above] {$\kappa_3$} (2W);
    \draw[->]    (2W) to node[above] {$\kappa_4$} (2Z);
\end{scope}
\end{tikzpicture}

\end{center}

The system $\mathcal{R}_{C,0}$ has a unique positive equilibrium that has a pair of purely imaginary eigenvalues when $\kappa_1 = \kappa_2+\kappa_3$. It was proven in \cite{bbh2023a} that when $\kappa_1 = \kappa_2+\kappa_3$ the system undergoes a degenerate Andronov--Hopf bifurcation. In fact, the Andronov--Hopf bifurcation is {\em vertical}, i.e., all periodic orbits appear simultaneously at $\kappa_1 = \kappa_2+\kappa_3$. Moreover, this degenerate bifurcation is unfolded transversely by the rate constants. Homogenisation of $\mathcal{R}_{C,0}$, corresponding to an instance of enlargement E3, leads to the system $\mathcal{R}_{C,1}$ which, due to Theorem~\ref{thmbifinherit}, must admit an Andronov--Hopf bifurcation. The positive stoichiometric classes of $\mathcal{R}_{C,1}$ are given by $x+y+z+w=c$ for $c>0$. Since the r.h.s.\ of the associated mass action ODE system is homogeneous (of degree two), the phase portrait does not depend on $c$. Direct computation shows that in each positive stoichiometric class there is a unique positive equilibrium. We cannot {\em a priori} decide on the basis of Theorem~\ref{thmbifinherit} whether the Andronov--Hopf bifurcation at that equilibrium is supercritical, subcritical or degenerate. 

In fact, one can confirm by direct computation that the \emph{first} Lyapunov coefficient in $\mathcal{R}_{C,1}$ can be negative, positive or zero, and thus, supercritical, subcritical and degenerate Andronov--Hopf bifurcations are {\em all} admitted. After fixing one of the rate constants (which does not restrict generality), the first Lyapunov coefficient vanishes along a curve in the three-dimensional parameter space. Moreover, the \emph{second} Lyapunov coefficient changes sign along this curve. Hence, both supercritical and subcritical Bautin bifurcations are admitted, and there is a single point in parameter space, where both the first and the second Lyapunov coefficients vanish. To understand the behaviour at and near that point in parameter space, one would need to compute the \emph{third} Lyapunov coefficient, a task we defer for future work. The computations of the first and the second Lyapunov coefficients are available in the Mathematica file at \cite{balazsgithub}.

\section{Discussion and conclusions}
\label{secconc}

When studying network dynamical systems, it is very natural to ask how properties of subnetworks affect those of the network. In the case of chemical reaction networks, there are a growing number of results, such as those in this paper, telling us how networks inherit dynamics from their subnetworks. Highlighting the need for such results is the fact that intuition is not always a good guide to which enlargements of a reaction network might preserve interesting behaviours, as illustrated by examples in papers including \cite{banajipanteaMPNE,banajiCRNosci,GrossJMB,Gutierrez2023}.

The practical importance of inheritance results is that direct computations on medium to large networks, such as those arising from biology, can be very challenging. Consequently, we need results which allow us to infer interesting behaviours in biologically relevant models from the behaviours occurring in simpler submodels. The results here on the inheritance of bifurcations, coupled with analytical and algorithmic work on bifurcations in small networks such as in \cite{banajiborosnonlinearity,banajiboroshofbauer1}, open up the possibility of identifying biologically relevant networks which display exotic behaviours in some parameter regions, without resorting to computations. 

A natural extension to this work would be to consider the inheritance of bifurcations in reaction networks with kinetics other than mass action. Indeed, some extensions in this direction are already evident in the proof of the main theorem. It is also likely that further enlargements exist which preserve bifurcations: see, for example, the generalisation of enlargement E2 proven in \cite{Cappelletti2020} to preserve nondegenerate multistationarity.

Another natural extension to this work involves going beyond local bifurcations of equilibria. We have remarked (see Remark~\ref{remgeneralise}) that the general set-up described here naturally implies such extensions. For example, we expect that results on the inheritance of bifurcations of periodic orbits, similar to those for equilibria, will hold. Writing down such results remains a task for future work.

\bibliographystyle{unsrt}

\bibliographystyle{unsrt}

\end{document}